\documentclass{artv3}
\usepackage{graphicx}
\usepackage[all]{xy}
\usepackage{fullpage}
\usepackage{amsmath,amsthm}
\usepackage{amssymb,amsfonts,amsbsy}
\usepackage{dsfont,mathrsfs}
\usepackage{eufrak}
\usepackage{color}
\usepackage[utf8]{inputenc}
\usepackage[T1]{fontenc}
\usepackage{enumitem}

\usepackage{cleveref}

\title{Exhaustion of $\cg{N}$ via rigid expansions}
\firstauthor{Jesús Hernández Hernández}
\date{}
\email{jhdez@matmor.unam.mx	}
\address{Centro de Ciencias Matemáticas, UNAM\newline
\indent Universidad Nacional Autónoma de México\newline
\indent Morelia, Mich. 58190\newline
\indent Mexico}
\urladdr{https://sites.google.com/site/jhdezhdez/}
\secondauthor{Cristhian E. Hidber}
\secondemail{hidber@matmor.unam.mx	}
\secondaddress{Centro de Ciencias Matemáticas, UNAM\newline
\indent Universidad Nacional Autónoma de México\newline
\indent Morelia, Mich. 58190\newline
\indent Mexico}
\secondurladdr{https://sites.google.com/view/cristhian-e-hidber}
\newtheorem{theorem}{Theorem}[section]
\newtheorem{lemma}[theorem]{Lemma}

\newtheorem{corollary}[theorem]{Corollary}

\newtheorem*{corollaryno}{Corollary}

\newtheorem{Theo}{Theorem}

\newtheorem{Coro}[Theo]{Corollary}

\newcommand*{\cg}[1]{\mathcal{C}(#1)}
\newcommand*{\mcg}[1]{\mathrm{Mod}(#1)}

\newcommand*{\adj}[2]{\mathrm{adj}_{#1}(#2)}
\newcommand*{\veps}{\varepsilon}
\newcommand*{\vkap}{\varkappa}
\newcommand*{\dt}[1]{\tau_{#1}}

\newcommand{\A}{\mathcal{A}}
\newcommand{\E}{\mathcal{E}}
\newcommand{\M}{\mathcal{M}}
\newcommand{\G}{\mathcal{G}}

\newcommand*{\Y}{\mathfrak{Y}}
\newcommand*{\X}{\mathfrak{X}}

\newcommand{\N}{\ensuremath{\mathbb{N}}}
\newcommand{\Z}{\ensuremath{\mathbb{Z}}}

\begin{document}
\maketitle
\begin{abstract}
 Let $N$ be a connected closed non-orientable surface of genus at least 6. In this work we prove that there exists a finite subgraph $\X$ (Irmak's finite rigid set \cite{Irmak2019}) such that any graph endomorphism $\varphi$ of $\cg{N}$ whose restriction to $\X$ is (locally) injective, $\varphi$ is induced by a homeomorphism of $N$. To prove this, we first prove that $\X$ and its rigid expansions exhaust $\cg{N}$.\\[0.5cm]
\end{abstract}
\section*{Introduction}

Let $N$ be a connected closed non-orientable surface of genus at least 6. In recent years, one of the most studied groups regarding $N$ is its \emph{mapping class group}, defined as the group of isotopy classes of self-homeomorphisms of $N$. In the orientable case this group has been studied extensively (see \cite{FarbMargalit} for a starting point), and in particular the rigidity phenomena present in this group is quite significant. To this end, one of the main tools of study is the \emph{curve graph of the surface}: given a surface, its curve graph is the simplicial graph whose vertices are the isotopy classes of essential simple closed curves and the edges are given by disjointness (see Subsection \ref{subsec:curvegraph} for more details). This combinatorial object was introduced by Harvey \cite{Harvey81} as an analogue in the context of the mapping class group and the Teichmüller space, for the Tits buildings in the study of arithmetic groups and homogeneous spaces.

Later on, Ivanov \cite{Ivanov} used the curve graph of an orientable surface to prove an analogue theorem to Margulis superrigidity, where a key result he proves is that (under the correct hypotheses) every automorphism of the curve graph is given by a homeomorphism of the surface; this result was later extended to the low-genus surfaces independently by Korkmaz \cite{Korkmaz} and Luo \cite{Luo}, and to non-orientable surfaces by Atalan-Korkmaz \cite{Atalan-Korkmaz}, Szepietowski \cite{Szepietowski20} and Stukow-Szepietowski \cite{SS25}.

Using these results as inspiration, there were several generalizations whose main ideas were proving algebraic rigidity results of the mapping class group via proving first combinatorial rigidity results of the curve graph. Examples of these are the works by Irmak \cite{Irmak2004} \cite{Irmak2012} \cite{Irmak2020}, by Behrstock-Margalit \cite{BM06}, by Shackleton \cite{Shackleton2007}, by Aramayona-Leininger \cite{AL13} \cite{AL16} and by the first author \cite{JHH1} \cite{JHH2} \cite{JHH3}. Of particular interest is the concept of \emph{rigid sets} from Aramayona-Leininger \cite{AL13} : a set $X$ of vertices of the curve graph of a surface is called \emph{rigid} if any locally injective map from $X$ to the curve graph is induced by the restriction to $X$ of a homeomorphism of the surface. This definition was extended by the first author \cite{JHH2} to the more general setting of simplicial graphs (see Subsection \ref{subsec:rigidexpansions}).

In this work, we prove a combinatorial rigidity result analogue to Theorem C by the first author \cite{JHH3}, but in the context of non-orientable surfaces.

\begin{Theo}\label{thm:combrigid}
    Let $N$ be a connected closed non-orientable surface of genus $g \geq 6$, and let $\varphi: \cg{N} \to \cg{N}$ be a graph morphism. If $\varphi|_{\X}$ is (locally) injective, then $\varphi$ is induced by a homeomorphism
\end{Theo}

To prove this result, we use rigid expansions and seed graphs: Given a simplicial connected graph $\Gamma$, a vertex $v$ and a non-empty set of vertices $A$, we say that \emph{$A$ uniquely determines $v$} if $v$ is the only vertex that is not in $A$ but it spans an edge with every element of $A$, and we denote it by $v = \langle A \rangle$. Given an induced subgraph $\Lambda$ of $\Gamma$, its first \emph{rigid expansion} is the induced subgraph whose vertex set is $V(\Lambda) \cup \{v \in V(\Gamma) : (\exists A \subset V(\Lambda)) \hspace{3mm} v = \langle A \rangle\}$. Then we can inductively define the $n$-th rigid expansion, denoted by $\Lambda^{n}$, and we define the $\omega$-th rigid expansion as the induced subgraph whose vertex set is $\bigcup_{n < \omega} \Lambda^{n}$. Finally, $\Lambda$ is a \emph{seed graph} if every vertex of $\Gamma$ is a vertex of some rigid expansion of $\Lambda$. See Subsection \ref{subsec:rigidexpansions} for details in a more general setting which applies to any countable ordinal.

To prove Theorem \ref{thm:combrigid}, we first prove that there exists a finite seed subgraph $\Y$ of $\cg{N}$.

\begin{Theo}\label{thm:exhaustionY}
 Let $N$ be a connected closed non-orientable surface of genus $g \geq 6$. There exists a finite seed subgraph $\Y$ of $\cg{N}$.
\end{Theo}

Using this, in Section \ref{sec:thmexhaustion} we prove that Irmak's finite rigid set $\X$ is also a seed subgraph.

\begin{Coro}\label{cor:exhaustionX}
 Let $N$ be a connected closed non-orientable surface of genus $g \geq 6$. There exists a finite rigid seed subgraph $\X$ of $\cg{N}$.
\end{Coro}

Then, Theorem \ref{thm:combrigid} follows from Corollary \ref{cor:exhaustionX} and the following result (where $Y^{\omega}$ denotes union of all finite rigid expansions of $Y$, see Subsection\ref{subsec:rigidexpansions}).

\begin{corollaryno}[C,\cite{JHH2}]
    Let $\Gamma$ be a connected simplicial graph, $Y$ be a rigid subgraph of $\Gamma$, and $\varphi: Y^{\omega} \to \Gamma$ be a graph morphism such that $\varphi|_{Y}$ is locally injective. Then $\varphi$ is the restriction to $Y^{\omega}$ of an automorphism of $\Gamma$, unique up to the pointwise stabilizer of $Y$ in $Aut(\Gamma)$.
\end{corollaryno}


\textbf{Acknowledgements:} 
 The first author was supported during the creation of this article by the research project grants UNAM-PAPIIT IA104620, IN102018 and IN114323. 
 The second author received support from a CONACYT Posdoctoral Fellowship and from UNAM-PAPIIT-IN105318.
 Both authors were supported during the creation of this article by the CONACYT Ciencia de Frontera 2019 research project grant CF2019 217392.
\section{Preliminaries}\label{sec:preliminaries}

In this section we introduce the terminology and results needed for this work in the subjects of rigid expansions in the context of simplicial graphs, the curve graph of a non-orientable surface, and the different models used for the non-orientable surfaces. Finally, we introduce the subgraph $\Y$ and some auxiliary curves that form the basis of the proofs of the following section.

\subsection{Rigid expansions}\label{subsec:rigidexpansions}

Let $\Gamma$ be a connected simplicial graph. We denote its vertex set as $V(\Gamma)$ and its edge set as $E(\Gamma)$.

If $v \in V(\Gamma)$, we denote by $\adj{\Gamma}{v}$ the set of all vertices in $\Gamma$ that are adjacent to $v$, i.e. it is the set $$\adj{\Gamma}{v}:= \{w \in V(\Gamma): \{v,w\} \in E(\Gamma)\}.$$

Let $v \in V(\Gamma)$ and $A \subset V(\Gamma)$. We say $v$ is \textit{uniquely determined by} $A$, denoted by $v = \langle A \rangle$, if it is the unique vertex in $\Gamma$ adjacent to all the vertices in $A$, i.e. $$\{v\} = \bigcap_{w \in A} \adj{\Gamma}{w}.$$

Recall that $\Lambda$ is an induced subgraph of $\Gamma$ if $V(\Lambda) \subset V(\Gamma)$ and $v,w \in V(\Lambda)$ span an edge in $\Lambda$ if and only if they span an edge in $\Gamma$. We denote that $\Lambda$ is an induced subgraph of $\Gamma$ by $\Lambda \leq \Gamma$. Hereafter, we always assume that any subgraph is an induced subgraph.

If $\Lambda \leq \Gamma$, we define its \textit{zeroth rigid expansion} as itself and denote it by $\Lambda^{0}$. Then, for any countable ordinal $\alpha$ we define the $\alpha$-th rigid expansion of $\Lambda$, denoted by $\Lambda^{\alpha}$ inductively as follows:
\begin{enumerate}
 \item If $\alpha$ is the successor of the ordinal $\beta$, then $\Lambda^{\alpha}$ is the subgraph of $\Gamma$ with vertex set: $$V(\Lambda^{\alpha}) := V(\Lambda^{\beta}) \cup \{v \in V(\Gamma) : (\exists A \subset V(\Lambda^{\beta})) \text{\hspace{3mm}} v = \langle A \rangle\}.$$
 \item If $\alpha$ is a limit ordinal, then $\Lambda^{\alpha}$ is the subgraph of $\Gamma$ with vertex set: $$V(\Lambda^{\alpha}) := \bigcup_{\beta < \alpha} V(\Lambda^{\beta}).$$
\end{enumerate}

Let $\Lambda \leq \Gamma$. We say that $\Lambda$ is a \emph{seed subgraph} if every vertex of $\Gamma$ is a vertex of some rigid expansion of $\Lambda$, i.e. $V(\Gamma) = \bigcup_{\alpha < \omega_1} V(\Lambda^{\alpha}).$

Below we describe a way to prove that a subgraph $\Lambda$ is a seed subgraph of $\Gamma$. This is based on the techniques used by the first author \cite{JHH1} for the curve graph of an orientable surface.

Let $G$ be a group acting by automorphisms on $\Gamma$, and let $\mathcal{G}$ be a symmetric generating set of $G$.

\begin{lemma}\label{lemma:gensetYinYk+1}
 Let $\Lambda \leq \Gamma$. Suppose there exists a countable ordinal $\alpha$ such that for all $g \in \mathcal{G}$ and all $v \in V(\Lambda)$ we have $g \cdot v \in V(\Lambda^{\alpha})$. Then, for all $w \in V(\Lambda^{1})$ and all $g \in \mathcal{G}$ we have that $g \cdot w \in V(\Lambda^{1 + \alpha})$.
\end{lemma}
\begin{proof}
 Let $w \in V(\Lambda^{1})$. If $w \in V(\Lambda)$ then the result follows. If $w \notin V(\Lambda)$, there exists a subset $A \subset V(\Lambda)$ such that $w = \langle A \rangle$.
 
 Since $G$ is acting by automorphisms we have that $g \cdot w = \langle g \cdot A \rangle$, for any $g \in G$.
 
 If $g \in \mathcal{G}$, we have that $g \cdot A \subset V(\Lambda^{\alpha})$ by hypothesis. Thus, $g \cdot w \in V(\Lambda^{1+\alpha})$ as desired.
\end{proof}

Using induction and the way the rigid expansions are defined we obtain the following corollary.

\begin{corollary}\label{cor:gensetYinYk+alpha}
 Let $\Lambda \leq \Gamma$. Suppose there exists a countable ordinal $\alpha$ such that for all $g \in \mathcal{G}$ and all $v \in V(\Lambda)$ we have $g \cdot v \in V(\Lambda^{\alpha})$. Then, for any countable ordinal $\beta$, for all $w \in V(\Lambda^{\beta})$ and all $g \in \mathcal{G}$ we have that $g \cdot w \in V(\Lambda^{\beta + \alpha})$.
\end{corollary}

If we denote by $|g|_{\mathcal{G}}$ the distance between $g \in G$ and its identity $e$ with respect to the word metric induced by $\mathcal{G}$, an application of Corollary \ref{cor:gensetYinYk+alpha}$|g|_{\mathcal{G}}$-many times implies the following result.

\begin{lemma}\label{lemma:groupYinY}
 Let $\Lambda \leq \Gamma$. Suppose there exists a countable ordinal $\alpha$ such that for all $g \in \mathcal{G}$ and all $v \in V(\Lambda)$ we have $g \cdot v \in V(\Lambda^{\alpha})$. Then, for any $g \in G \setminus \{e\}$ and any $v \in V(\Lambda^{\beta})$ we have that $g \cdot v \in V(\Lambda^{\beta + \alpha \cdot |g|_{\mathcal{G}}})$. In particular, $G \cdot v \subset V(\Lambda^{\beta + \alpha \cdot \omega})$.
\end{lemma}

As a direct consequence of this, we have the following theorem.

\begin{theorem}\label{thm:seedgraph}
 Let $\Lambda \leq \Gamma$ and let also $\{v_{i}\}_{i \in I}$ be a set of representatives of $G \backslash \Gamma$. Suppose there exists a countable ordinal $\alpha$ such that for all $g \in \mathcal{G}$ and all $v \in V(\Lambda)$ we have $g \cdot v \in V(\Lambda^{\alpha})$. If for every $i \in I$ there exists a countable ordinal $\beta_{i}$ such that $v_{i} \in V(\Lambda^{\beta_{i}})$, then $\Lambda$ is a seed subgraph of $\Gamma$.
\end{theorem}

\subsection{Curve graph}\label{subsec:curvegraph}

We denote the closed non-orientable surface of genus $g$ as $N_{g}$. Let $N = N_g$ for $g \geq 6$.

A \textit{curve} is a topological embedding of the unit circle into $N$. We often abuse notation and call ``curve'' the embedding, its image on $N$ or its isotopy class. The context makes clear which use we mean.

A curve is \textit{essential} if it is not isotopic to a boundary curve and if it does not bound a disk, a punctured disk or a Möbius band. Unless otherwise stated, all curves are assumed to be essential.

The \textit{(geometric) intersection number} of two isotopy classes of essential curves $\alpha$ and $\beta$ is defined as: $$i(\alpha,\beta):= \min \{|a \cap b| : a \in \alpha, b \in \beta\}.$$

A curve $\alpha$ is \textit{separating} if $N \setminus \alpha$ is disconnected. It is \textit{non-separating} otherwise.

A curve $\alpha$ is \textit{one-sided} if its regular neighbourhood is homeomorphic to a Möbius band. It is \emph{two-sided} otherwise.

The \textit{curve graph} of $N$ is defined as the simplicial graph whose vertices are the isotopy classes of essential curves on $N$, and two vertices $[\alpha]$ and $[\beta]$ span an edge if $i([\alpha],[\beta]) = 0$.

\subsection{Model for $N_{g}$ and the subgraph $\Y \leq \cg{N}$}\label{subsec:model}

Let $H$ and $H^{\prime}$ be a pair of $2g$-gons with sides labeled $(a_{1}, e_{1}, \ldots, a_{g}, e_{g})$ and $(a^{\prime}_{1}, e^{\prime}_{1}, \ldots, a^{\prime}_{g}, e^{\prime}_{g})$ respectively. We then define $P$ as: $$P := H \sqcup H^{\prime}/ \sim,$$ where $\sim$ is the gluing of the sides $e_{i}$ with $e^{\prime}_{i}$ for all $i$. Note that $P$ is homeomorphic to a surface of genus zero and $g$ boundary components.

Finally we define $N$ as $P / \sim$ where $\sim$ is now the gluing of the sides $a_{i}^{-1}$ with $a^{\prime}_{i}$ for all $i$. Thus, $N$ is homeomorphic to $N_{g}$.

Now, we define the curves that make up the vertices of $\Y$:
\begin{itemize}
 \item Let $1 \leq i, j \leq g$. We define $a_{i,j}$ and $a^{\prime}_{i,j}$ as the linear segments with endpoints $(a_{i}, e_{j})$ and $(a^{\prime}_{i}, e^{\prime}_{j})$ respectively. Then, we define the curve $\alpha_{i,j}$ as the curve in $N$ obtained by projecting $a_{i,j} \cup a^{\prime}_{i,j}$. See Figure \ref{fig:def-alphaij}. Note that $\alpha_{i,j}$ is always one-sided.
 \begin{figure}[ht]
  \centering
  \includegraphics[height=4cm]{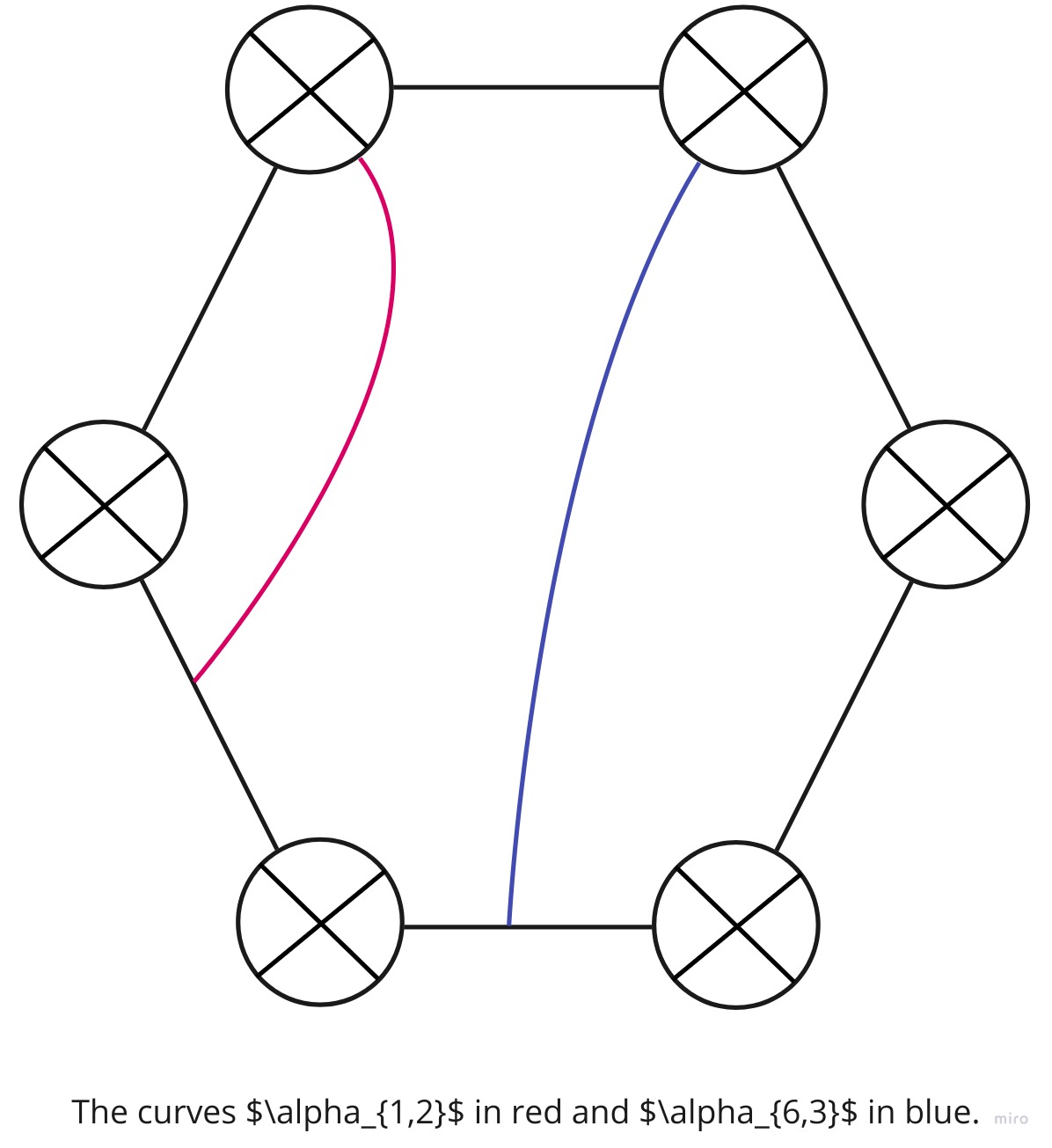}\caption{The curves $\alpha_{1,2}$ in red and $\alpha_{6,3}$ in blue.}\label{fig:def-alphaij}
 \end{figure}
 \item Let $1 \leq i,j \leq g$ with $|i-j| \geq 2$ modulo $g$. We define $n^{+}_{i,j}$ as a pair of linear segments on $H$, one going from $e_{i}$ to $e_{j-1}$ and the other from $e_{i-1}$ to $e_{j}$. We define $(n^{+})^{\prime}_{i,j}$ as a pair of linear segments on $H^{\prime}$, one going from $e^{\prime}_{i}$ to $e^{\prime}_{i-1}$ and the other from $e^{\prime}_{j-1}$ to $e^{\prime}_{j}$. We define $\nu^{+}_{i,j}$ as the curve in $N$ obtained by projecting $n^{+}_{i,j} \cup (n^{+})^{\prime}_{i,j}$. See Figure \ref{fig:def-nupmij}. Note that $\nu^{+}_{i,j}$ is always two-sided.
 \item Let $1 \leq i,j \leq g$ with $|i - j| \geq 2$ modulo $g$. We define $n^{-}_{i,j}$ as a pair of linear segments on $H^{\prime}$, one going from $e^{\prime}_{i}$ to $e^{\prime}_{j-1}$ and the other from $e^{\prime}_{i-1}$ to $e^{\prime}_{j}$. We define $(n^{-})^{\prime}_{i,j}$ as a pair of linear segments on $H$, one going from $e_{i}$ to $e_{i-1}$ and the other from $e_{j-1}$ to $e_{j}$. We define $\nu^{-}_{i,j}$ as the curve in $N$ obtained by projecting $n^{-}_{i,j} \cup (n^{-})^{\prime}_{i,j}$. See Figure \ref{fig:def-nupmij}. Note that $\nu^{-}_{i,j}$ is always two-sided.
 \begin{figure}[ht]
  \centering
  \includegraphics[height=4cm]{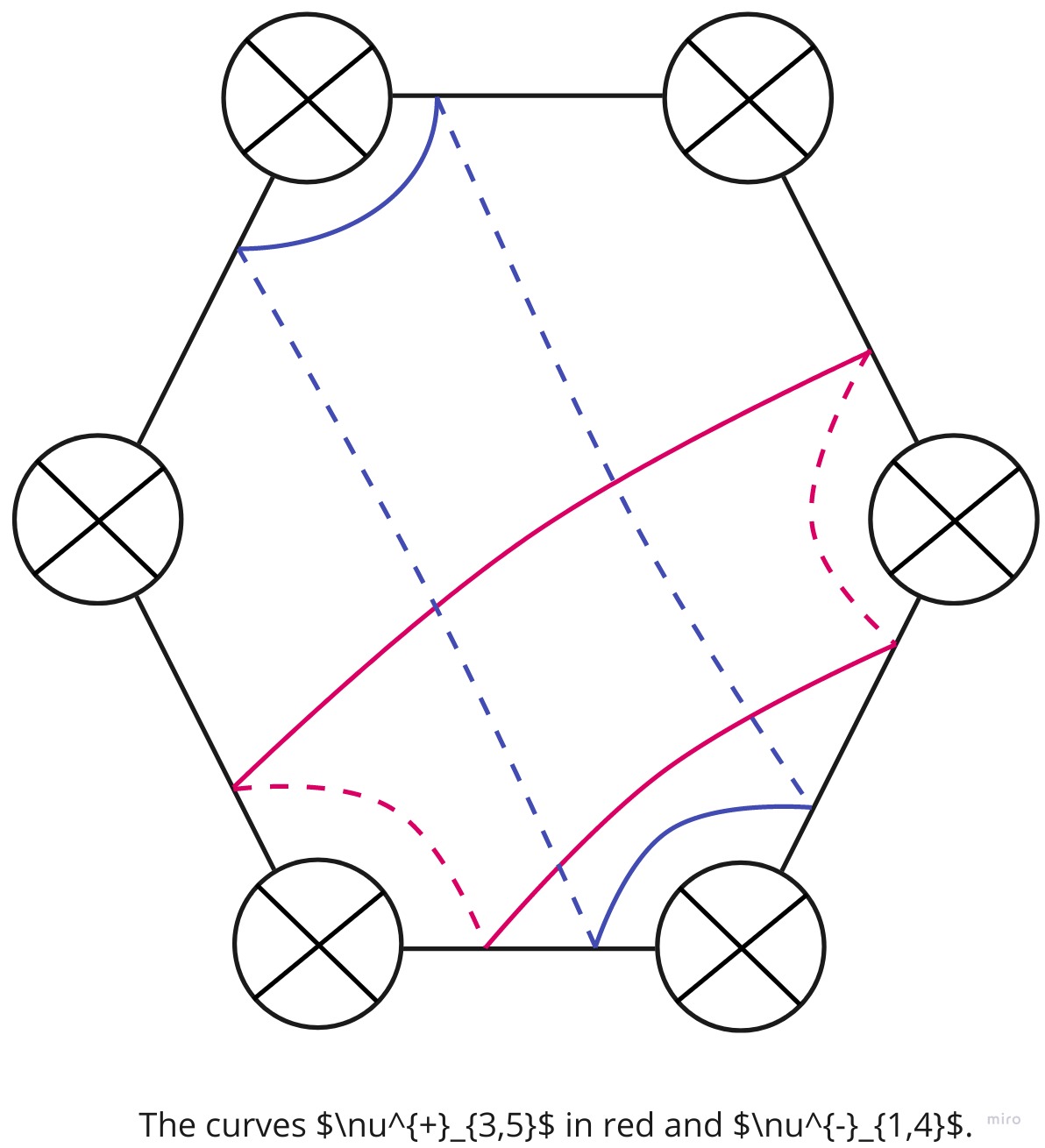}\caption{The curves $\nu^{+}_{3,5}$ in red and $\nu^{-}_{1,4}$ in blue.}\label{fig:def-nupmij}
 \end{figure}
\end{itemize}

Note that $\alpha_{i,i}$ and $\alpha_{i,i-1}$ are isotopic. Thus, to ease notation, we define $\alpha_{i}$ as $\alpha_{i,i}$ (isotopic to $\alpha_{i,i-1}$).

We define the sets $\A$ and $\M$ as follows (the subindices are modulo $g$): $$\A := \{\alpha_{i,i+1}\}_{1 \leq i \leq g} \cup \{\alpha_{i,i-2}\}_{1 \leq i \leq g},$$
$$\M := \{\nu^{-}_{5,7}\}.$$

Finally, we define $\Y$ as the subgraph of $\cg{N}$ whose vertex set is $\A \cup \M$.

\subsection{Useful additional curves}\label{subsec:usefulcurves}

In this subsection we uniquely determine several curves that become quite useful in proving Theorem \ref{thm:exhaustionY}.

\subsubsection{The curves $\veps_{i,j}$ and the set $\E$}\label{subsubsec:vepsijE}

Let $1 \leq i,j \leq g$. We define $e_{i,j}$ and $e^{\prime}_{i,j}$ as the linear segments with endpoints $(e_{i},e_{j})$ and $(e^{\prime}_{i}, e^{\prime}_{j})$ respectively. Then, we define the curve $\veps_{i,j}$ as the curve in $N$ obtained by projecting $e_{i,j} \cup e^{\prime}_{i,j}$. See Figure \ref{fig:def-vepsij}. Note that $\veps_{i,j}$ is always two-sided.
\begin{figure}[ht]
 \centering
 \includegraphics[height=4cm]{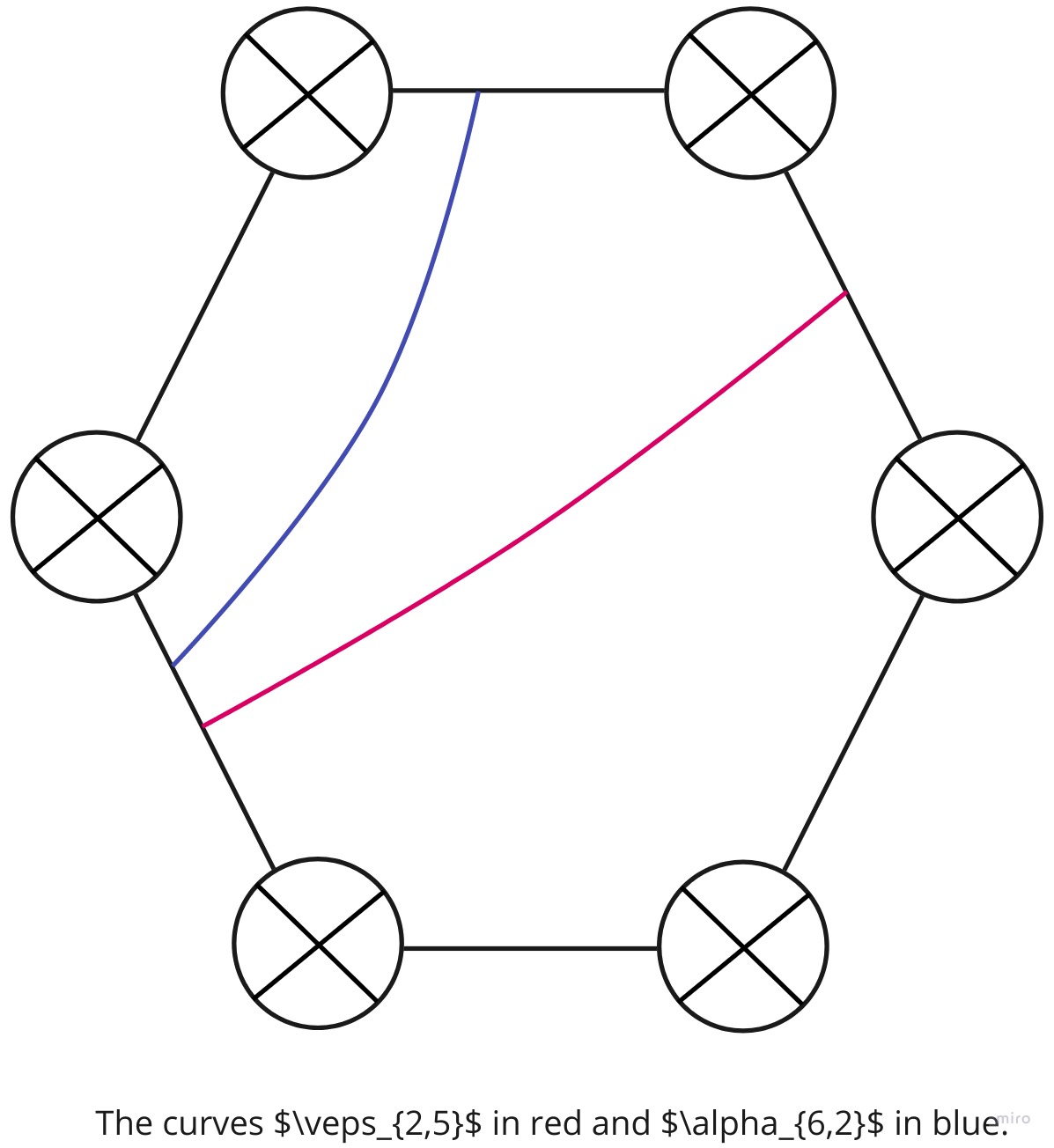}\caption{The curves $\veps_{2,5}$ in red and $\veps_{6,2}$ in blue.}\label{fig:def-vepsij}
\end{figure}

\begin{lemma}\label{lemma:vepsijv2}
 For all $1 \leq i,j \leq g$ such that $\veps_{i,j}$ is essential, we have that $\veps_{i,j} \in V(\Y^{1})$.
\end{lemma}
\begin{proof}
If $|i-j| \leq 1$, then $\veps_{i,j}$ is not essential. Thus, we assume that $|i-j| \geq 2$.

We define the following sets: $$A_{1} = \{\alpha_{i+1,i+2}, \ldots, \alpha_{j-1,j}\},$$ $$A_{2} = \{\alpha_{j,j-2}, \ldots, \alpha_{i+2,i}\},$$ $$A_{3} = \{\alpha_{j+1,j+2}, \ldots, \alpha_{i-1,i}\},$$ $$A_{4} = \{\alpha_{i,i-2}, \ldots, \alpha_{j+2,j}\}.$$ Then we have that $\veps_{i,j} = \langle A_{1} \cup A_{2} \cup A_{3} \cup A_{4} \rangle \in V(\Y^{1})$. See Figure \ref{fig:lemma-vepsijv2}.
\end{proof}

\begin{figure}[ht]
 \centering
 \includegraphics[height=4cm]{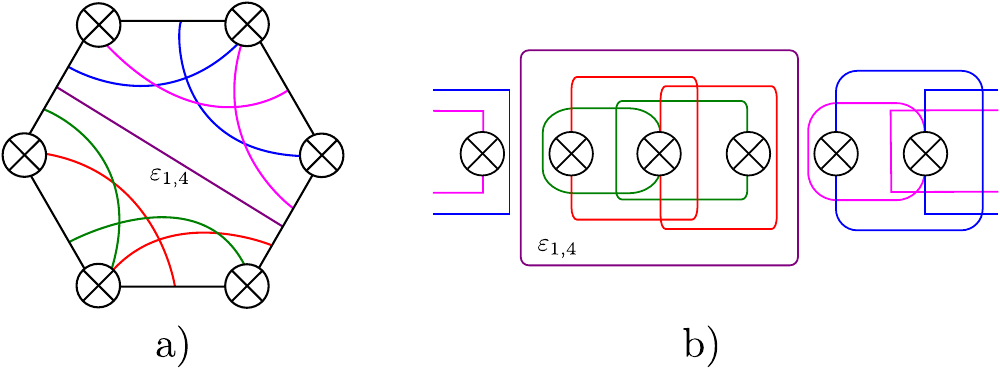}\caption{The curve $\varepsilon_{1,4}$ in purple.}\label{fig:lemma-vepsijv2}
\end{figure}

Due to its usefulness, we define the following set: $$\E:= \{\veps_{i,i+2}\}_{1 \leq i \leq g} \subset V(\Y^{1}).$$ We use this set thoroughly through the text.

\subsubsection{The curves $\alpha_{i}$}\label{subsubsec:alphai}

Here we prove the following lemma:

\begin{lemma}\label{lemma:alphai}
 For all $1 \leq i \leq g$, we have that $\alpha_{i} \in V(\Y^{1})$.
\end{lemma}
\begin{proof}
 Let $1 \leq i \leq g$ be fixed. We define the sets: $$A_{1} = \{\alpha_{i+1,i+2}, \ldots, \alpha_{i-2,i-1}\},$$ $$A_{2} = \{\alpha_{i-1,i-3}, \ldots, \alpha_{i+2,i}\}.$$ Then, $\alpha_{i} = \langle A_{1} \cup A_{2} \rangle \in V(\Y^{1})$. See Figure \ref{fig:lemma-alphai}.
\end{proof}

\begin{figure}[ht]
 \centering
 \includegraphics[height=4cm]{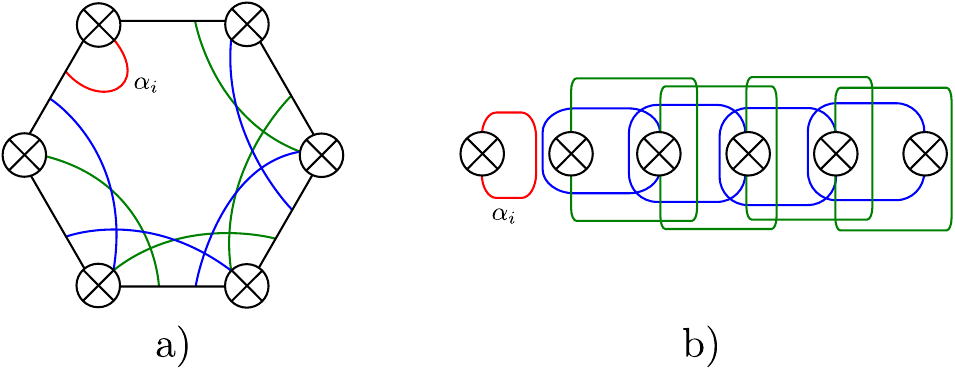}\caption{The curve $\alpha_{1}$ in red.}\label{fig:lemma-alphai}
\end{figure}

\subsubsection{The curves $\mu^{1}_{i}$}\label{subsubsec:mui}

In this subsubsection we define a family of curves $\mu^{1}_{i}$ for $1 \leq i \leq g$. There is a superscript $1$ because these curves end up being part of a slightly larger family $\{\mu^{m}_{i}: 1 \leq i \leq g, 1 \leq m \leq 2\}$ (see Subsubsection \ref{subsubsec:dtgammaiA}).

For $1 \leq i \leq g$ we define $\mu^{1}_{i}$ as follows: Let $$E_{1} = \E \setminus \{\veps_{i-3,i-1},\veps_{i-2,i}, \veps_{i-1,i+1},\veps_{i,i+2}\} \subset V(\Y^{1})$$ $$E_{2} = \{\veps_{i-2,i+1}\} \subset V(\Y^{1}),$$ $$A_{1} = \{\alpha_{i,i+1}, \alpha_{i,i-2}\} \subset V(\Y),$$ $$A_{2} = \{\alpha_{j}\}_{j \neq i} \subset V(\Y^{1}).$$ Then, $\mu^{1}_{i}:= \langle E_{1} \cup E_{2} \cup A_{1} \cup A_{2}\rangle \in V(\Y^{2})$. See Figure \ref{fig:def-mui}. 
\begin{figure}[ht]
  \centering
  \includegraphics[height=4cm]{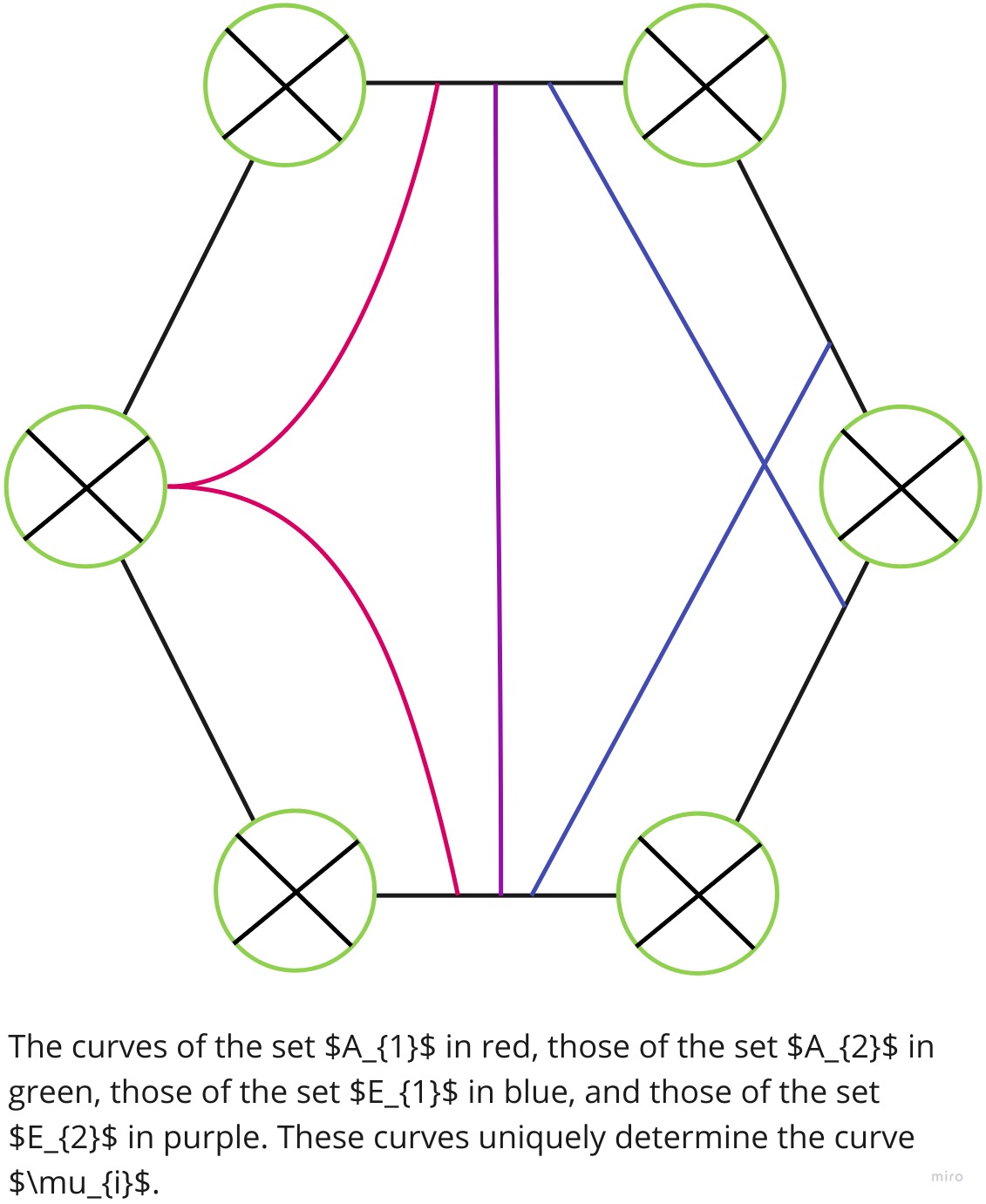}\caption{The curves of the set $A_{1}$ in green, those of the set $A_{2}$ in red, those of the set $E_{1}$ in blue, and those of the set $E_{2}$ in purple. These curves uniquely determine the curve $\mu^{1}_{2}$.}\label{fig:def-mui}
 \end{figure}

\subsubsection{The curves $\alpha_{i,j}$}\label{subsubsec:alphaij}

\begin{lemma}\label{lemma:alphaij}
 For all $1 \leq i,j \leq g$, we have that $\alpha_{i,j} \in V(\Y^{3})$.
\end{lemma}
\begin{proof}
 Let $1 \leq i,j \leq g$ be fixed. If $j \in \{i-2,i-1,i,i+1\}$ then $\alpha_{i,j}$ is either isotopic to $\alpha_{i}$ or is an element of $\A$. Thus, $\alpha_{i,j} \in V(\Y^{1})$.
 
 Otherwise, we define the following sets: $$E_{1} = \{\veps_{i,i+2}, \ldots, \veps_{j-2,j}\} \subset V(\Y^{1})$$ $$E_{2} = \{\veps_{j,j+2}, \ldots, \veps_{i-3,i-1}\} \subset V(\Y^{1})$$ $$E_{3} = \{\veps_{i,j}, \veps_{i-1,j}\} \subset V(\Y^{1}),$$ $$A = \{\alpha_{j}\}_{i \neq j} \subset V(\Y^{1}),$$ $$M= \{\mu^{1}_{i}\} \subset V(\Y^{2}).$$
 
 Then, we have that $\alpha_{i,j}=\langle E_{1} \cup E_{2} \cup E_{3} \cup A \cup M \rangle \in V(\Y^{3})$. See Figure \ref{fig:udet-alphaij}.
\begin{figure}[ht]
  \centering
  \includegraphics[height=4cm]{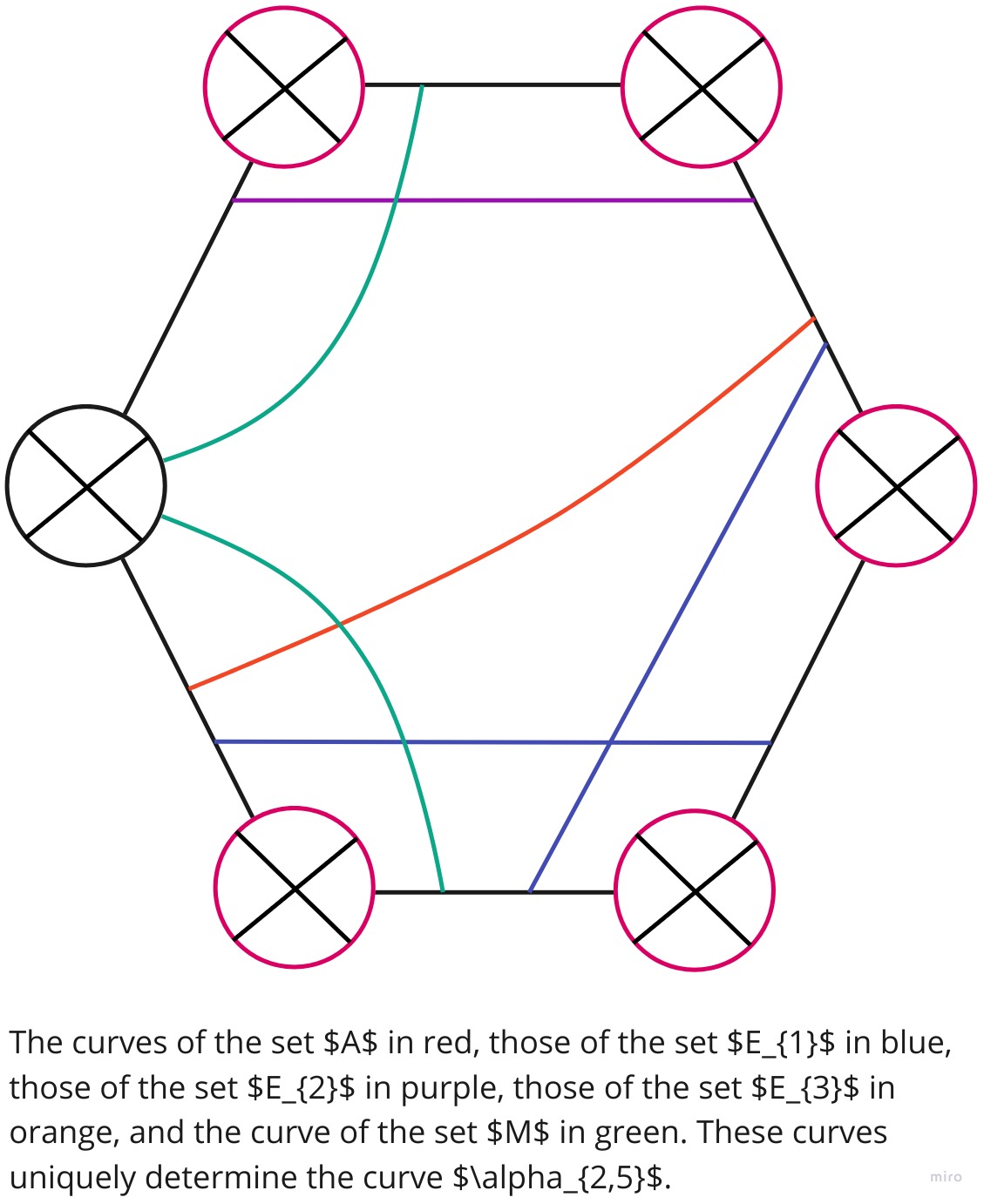}\caption{The curves of the set $A$ in red, those of the set $E_{1}$ in blue, those of the set $E_{2}$ in purple, those of the set $E_{3}$ in orange, and the curve of the set $M$ in green. These curves uniquely determine the curve $\alpha_{2,5}$ in fuchsia.}\label{fig:udet-alphaij}
 \end{figure}
\end{proof}

\subsubsection{The curves $\nu^{\pm}_{i,j}$}\label{subsubsec:nupmij}

\begin{lemma}\label{lemma:nupmij}
 For all $1 \leq i,j \leq g$ with $|i-j|\geq 2$ (modulo $g$), we have that $\nu^{+}_{i,j}, \nu^{-}_{i,j} \in V(\Y^{6})$.
\end{lemma}
\begin{proof}
Let $1 \leq i,j \leq g$ with $|i-j|\geq 2$ (modulo $g$). In general, we prove that $\nu^{\pm}_{i,j}$ is uniquely determined by a set $A(i,j) \subset V(\Y^{1})$ and an auxiliary curve $\nu$. The former depends only on $i,j$ and the latter depends also of the sign of $\nu^{\pm}_{i,j}$. As such, we describe the set $A(i,j)$, and then use a table to point out what curve is $\nu$.

We define $A(i,j)$ as follows: $$A(i,j) = (\E \setminus \{\veps_{i-2,i},\veps_{i-1,i+1},\veps_{j-2,j},\veps_{j-1,j+1}\}) \cup \{\veps_{i,j-1},\veps_{i-1,j}\} \cup \{\alpha_{k} : 1 \leq k \leq g\} \subset V(\Y^{1}).$$

Before defining $\nu$, we need the following definition to facilitate the exposition.

Let $\mathfrak{i}: \{1, \ldots, g\} \to \mathbb{S}^{1}$ be defined as $k \mapsto \exp(\frac{2k\pi \sqrt{-1}}{g})$. Let also $1 \leq i,j,x,y \leq g$ with $|i-j|\geq 2$ (modulo $g$). We say that $x$ and $y$ are \textit{in different sides from} $i$ and $j$ if $\mathfrak{i}(x)$ and $\mathfrak{i}(y)$ are in different connected components of $\mathbb{S}^{1} \setminus \{\mathfrak{i}(i), \mathfrak{i}(j)\}$; we denote this as $\{x,y\}$ \emph{idsf} $\{i,j\}$.

We describe the curve $\nu$ in Table \ref{table:nu}, along with the rigid expansion to which $\nu^{\pm}_{i,j}$ belongs:

\begin{table}[ht]
 \centering
 \begin{tabular}{|c|c|c|}
  \hline
  $\nu^{\pm}_{i,j}$ & The curve $\nu$ & $m$ for $\nu^{\pm}_{i,j} \in V(\Y^{m})$\\ \hline
  $\nu^{+}_{i,6}$ & $\nu^{-}_{5,7}$ & $2$\\ \hline
  $\nu^{-}_{i,j}$ with $i \neq 6 \neq j$ & $\nu^{+}_{k,6}$ with $\{k,6\}$ \emph{idsf} $\{i,j\}$ & $3$\\ \hline
  $\nu^{+}_{i,j}$ with $\{i,j\} \neq \{5,7\}$ & $\nu^{-}_{k,l}$ with $\{k,l\}$ \emph{idsf} $\{i,j\}$ and $k \neq 6 \neq l$ & $4$\\ \hline
  $\nu^{-}_{i,6}$ & $\nu^{+}_{k,l}$ with $\{k,l\}$ \emph{idsf} $\{i,j\}$ and $\{k,l\}\neq\{5,7\}$ & $5$\\ \hline
  $\nu^{+}_{5,7}$ & $\nu^{-}_{k,l}$ with $\{k,l\}$ \emph{idsf} $\{5,7\}$ & $6$\\ \hline
 \end{tabular}
 \caption{The curves $\nu$ needed to unique determine $\nu^{\pm}_{i,j}$.}\label{table:nu}
\end{table}

See Figures \ref{fig:udet-nupmij}. Therefore, $\nu^{\pm}_{i,j} \in V(\Y^{6})$.
\begin{figure}[ht]
  \centering
  \includegraphics[height=4cm]{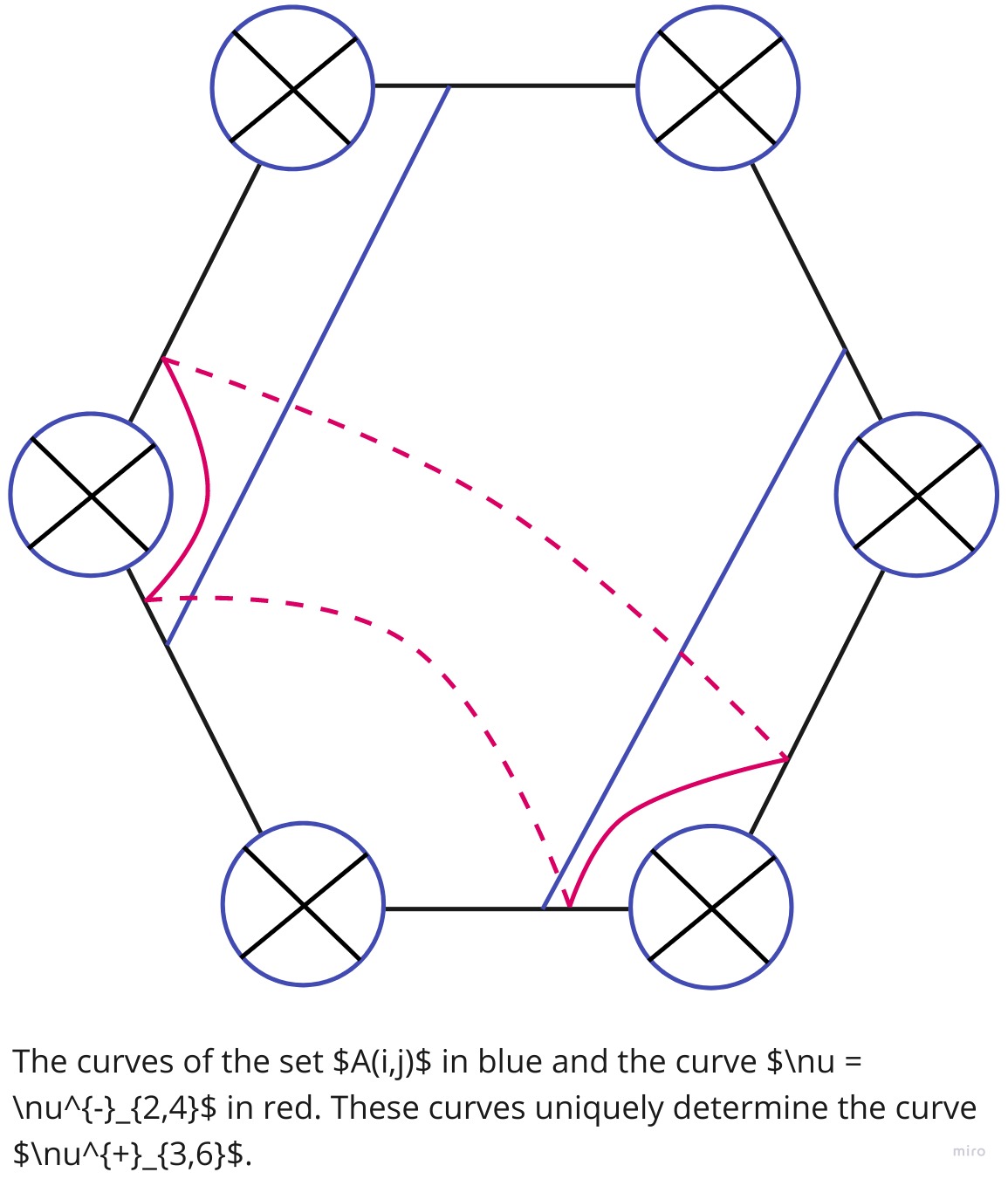}\caption{The curves of the set $A(i,j)$ in blue and the curve $\nu = \nu^{-}_{2,4}$ in red. These curves uniquely determine the curve $\nu^{+}_{3,6}$ in green.}\label{fig:udet-nupmij}
 \end{figure}
\end{proof}

\section{Translations of $\Y$ using generators}\label{sec:GenSetY}

In this section we review Szepietowski's generating set for $\mcg{N}$ (see \cite{Szepietowski}), and see that (after symmetrising this generating set) any translation of $\Y$ by elements of the generating set is contained in some rigid expansion of $\Y$ (see Lemmata \ref{lemma:dtgammaiA}, \ref{lemma:dtgammaiM}, \ref{lemma:y(Y)} and \ref{lemma:taubeta(Y)}). This last part is to apply Lemma \ref{lemma:groupYinY}.

\subsection{Szepietowski generating set}\label{subsec:GenSet}

In Theorem 3 of \cite{Szepietowski}, Szepietowski gives a generating set of $\mcg{N}$, but to present it we need to define some curves and notations.

Given a two-sided curve $\alpha$, we denote by $\dt{\alpha}$ the (right) Dehn twist along $\alpha$.

Let $K$ be a Klein bottle embedded in $N$ with boundary curve $\eta$. Since $N = N_{g}$ with $g \geq 5$, we have that $\eta$ is always essential. We denote by $y_{K}$ the Y-homeomorphism in $K$ introduced in \cite{Lickorish}, extended as the identity in the complement of $N$. Note that $y_{K}^{2} = \dt{\eta}$. Another way to describe $y_{K}$ is to say that $y_{K}$ is the crosscap slide of one of the crosscaps in $K$ into the other. See Figure \ref{fig:def-Y-Homeomorphism}.

\begin{figure}[ht]
  \centering
  \includegraphics[height=3.5cm]{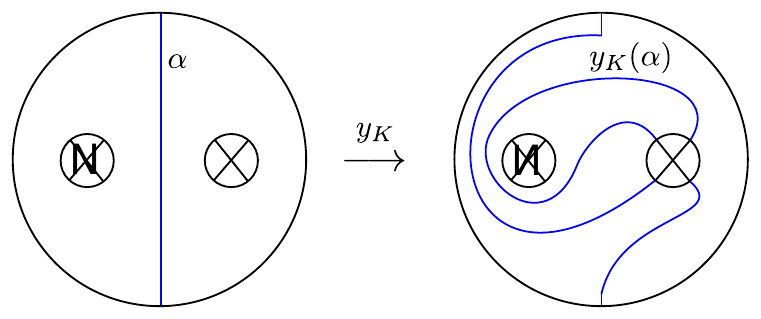}\caption{The Y-homeomorphism.}\label{fig:def-Y-Homeomorphism}
 \end{figure}

For each $1 \leq i \leq g$, we define the $c_{i}$ and $c^{\prime}_{i}$ as the linear segments with endpoints $(a_{i},a_{i+1})$ and $(a^{\prime}_{i}, a^{\prime}_{i+1})$ respectively (subindices modulo $g$). Then, we define the curve $\gamma_{i}$ as the curve in $N$ obtained by projecting $c_{i} \cup c^{\prime}_{i}$. See Figure \ref{fig:def-gammaibeta}. Note that $\gamma_{i}$ is always two-sided.

Let $b$ be the linear segment in $H$ with endpoints $((e_{g})_{+}, (e_{4})_{-})$ where $( \cdot )_{-}$ denotes the origin point and $( \cdot )_{+}$ denotes the terminal point of the respective segments. Then, we define the curve $\beta$ as the curve in $N$ obtained by projecting $e_{1} \cup e_{2} \cup e_{3} \cup b$. See Figure \ref{fig:def-gammaibeta}. Note that $\beta$ is a two-sided curve.

\begin{figure}[ht]
  \centering
  \includegraphics[height=4cm]{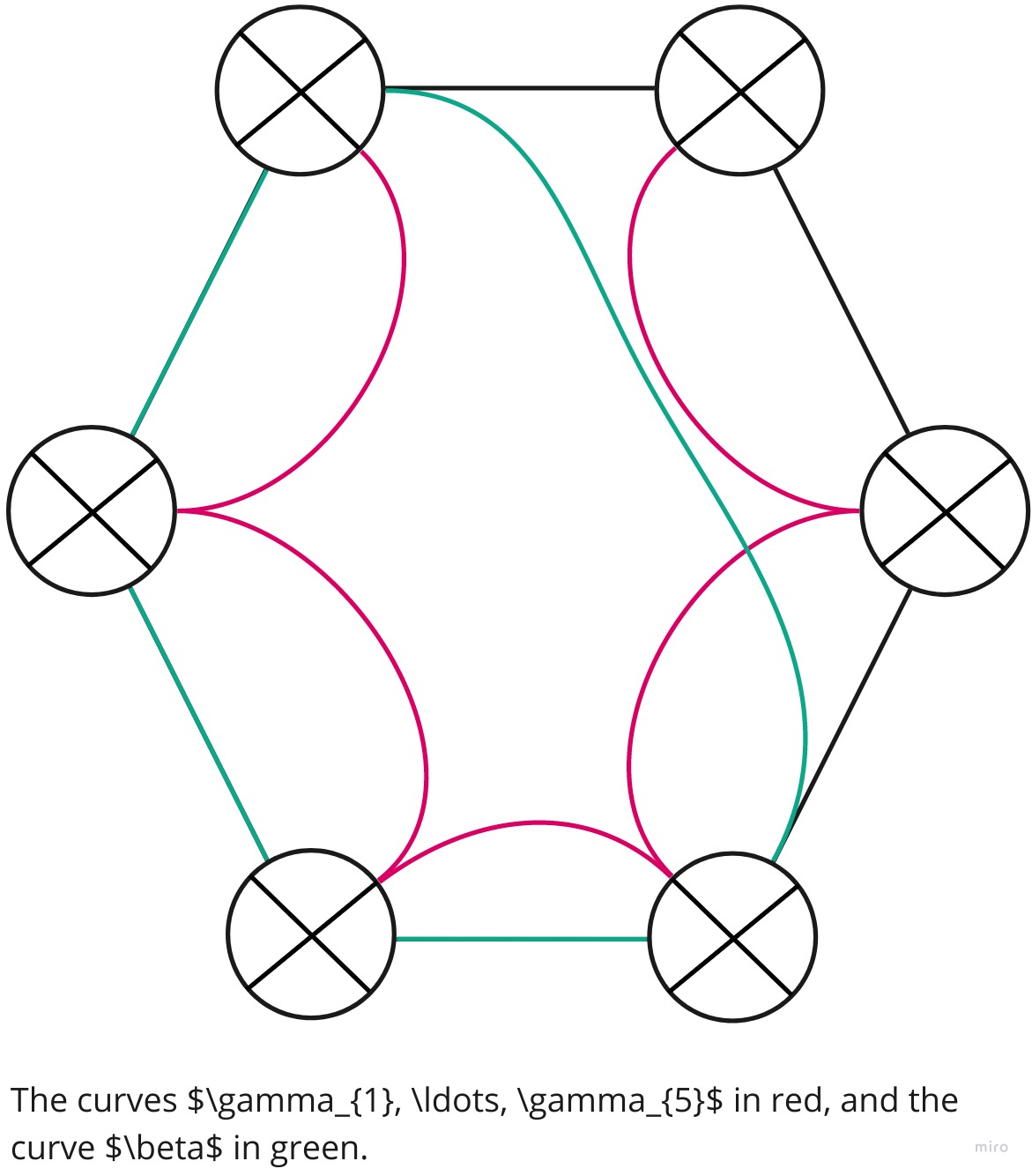}\caption{The curves $\gamma_{1}, \ldots, \gamma_{5}$ in red, and the curve $\beta$ in green.}\label{fig:def-gammaibeta}
 \end{figure}

Szepietowski proved the following result, presented here in the context of $g \geq 5$ and using our notation:

\begin{theorem}[3 in \cite{Szepietowski}]\label{thm:SzepietowskiGenSet}
 Let $K$ be the Klein bottle bounded by $\veps_{g-2,g}$, and $y = y_{K}$ be the corresponding Y-homeomorphism. If $g \geq 5$, the set $$C = \{\dt{\gamma_{i}} : 1 \leq i \leq g-1\} \cup \{\dt{\beta}\} \cup \{y\},$$ is a generating set of $\mcg{N_{g}}$.
\end{theorem}

We define $\G$ as the symmetrised generating set of $\mcg{N}$ obtained from $C$.

\subsection{$\dt{\gamma_{i}}^{\pm 1}(\Y) \subset \Y^{11}$}\label{subsec:dtgammaiY}

In this subsection we prove the following lemmas.

\begin{lemma}\label{lemma:dtgammaiA}
 For all $1 \leq i \leq g-1$ we have that $\dt{\gamma_{i}}^{\pm 1}(\A) \subset V(\Y^{8})$.
\end{lemma}

\begin{lemma}\label{lemma:dtgammaiM}
 For all $1 \leq i \leq g-1$ we have that $\dt{\gamma_{i}}^{\pm 1}(\M) \subset V(\Y^{11})$.
\end{lemma}

Given the lengths of these proofs, we dedicate a subsubsection to each one.

\subsubsection{Proof of Lemma \ref{lemma:dtgammaiA}}\label{subsubsec:dtgammaiA}

Given that $\dt{\gamma_{i}}^{\pm 1}(\alpha_{j,k}) = \alpha_{j,k}$ if $i(\alpha_{j,k}, \gamma_{i}) = 0$, to prove Lemma \ref{lemma:dtgammaiA} we only need to verify that $\dt{\gamma_{i}}^{\pm 1}(\alpha_{i+1,i+2})$, $\dt{\gamma_{i}}^{\pm 1}(\alpha_{i,i+1})$, $\dt{\gamma_{i}}^{\pm 1}(\alpha_{i-1,i})$, $\dt{\gamma_{i}}^{\pm 1}(\alpha_{i,i-2})$, $\dt{\gamma_{i}}^{\pm 1}(\alpha_{i+1,i-1})$, $\dt{\gamma_{i}}^{\pm 1}(\alpha_{i+2,i}) \in V(\Y^{6})$. However, to do this we need several auxiliary curves and lemmata.

Let $1 \leq j \leq g$ be fixed, and let $$E = (\E \setminus \{\veps_{j-3,j-1}, \veps_{j-2,j}, \veps_{j-1,j+1}, \veps_{j,j+2}, \veps_{j+1,j+3}\}) \cup \{\veps_{j-2,j+2}\} \subset V(\Y^{1}),$$ $$A_{1} = \{\alpha_{k}: k \neq j, j+1\} \subset V(\Y^{1}),$$ $$A_{2} = \{\alpha_{j,j+2}, \alpha_{j+1,j-2}\} \subset V(\Y^{3}).$$
 Then we define $\mu^{2}_{j}$ as follows: $$\mu^{2}_{j} := \langle E \cup A_{1} \cup A_{2} \rangle \in V(\Y^{4}).$$ See Figure \ref{fig:def-mujup2}.

\begin{figure}[ht]
  \centering
  \includegraphics[height=3.5cm]{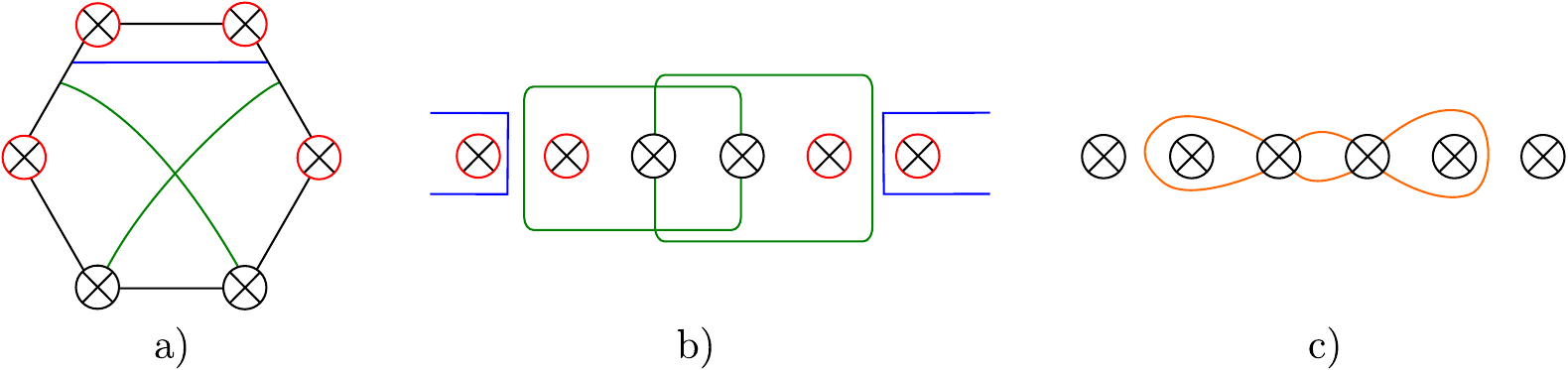}\caption{(a) and (b) depict the curves (from two angles) necessary to uniquely determine $\mu_{j}^{2}$, which is depicted in (c).}\label{fig:def-mujup2}
 \end{figure}

\begin{lemma}\label{lemma:dtgammaiAp1}
 For all $1 \leq i \leq g-1$ we have that $\dt{\gamma_{i}}^{\pm 1}(\alpha_{i+1,i+2})$, $\dt{\gamma_{i}}^{\pm 1}(\alpha_{i,i-2}) \in V(\Y^{5})$.
\end{lemma}
\begin{proof}
 To see that $\dt{\gamma_{i}}(\alpha_{i+1,i+2}) \in V(\Y^{5})$, let $$E = (\E \setminus \{\veps_{i-2,i}, \veps_{i-1,i+1}, \veps_{i,i+2}, \veps_{i+1,i+3}\}) \cup \{\veps_{i-1,i+2}\} \subset V(\Y^{1}),$$ $$A_{1} = \{\alpha_{k}: k \neq i,i+1\} \subset V(\Y^{1}),$$ $$A_{2} = \{\alpha_{i+1,i+2}, \alpha_{i+1,i-1}\} \subset V(\Y),$$ $$M = \{\mu^{2}_{i}\} \subset V(\Y^{4}).$$ Then we have that $\dt{\gamma_{i}}(\alpha_{i+1,i+2}) = \langle E \cup A_{1} \cup A_{2} \cup M\rangle \in V(\Y^{5})$; see Figure \ref{fig:lemma-dtgammaiAp1-alphai+1i+2}. Analogously, we also have that $\dt{\gamma_{i}}^{-1}(\alpha_{i,i-2}) \in V(\Y^{5})$.
 
 Given that $\dt{\gamma_{i}}^{-1}(\alpha_{i+1,i+2}) = \alpha_{i,i+2}$, we have that $\dt{\gamma_{i}}^{-1}(\alpha_{i+1,i+2}) \in V(\Y^{3})$. Analogously, we also have that $\dt{\gamma_{i}}(\alpha_{i,i-2}) \in V(\Y^{3})$.
\end{proof}

\begin{figure}[ht]
  \centering
  \includegraphics[height=3.5cm]{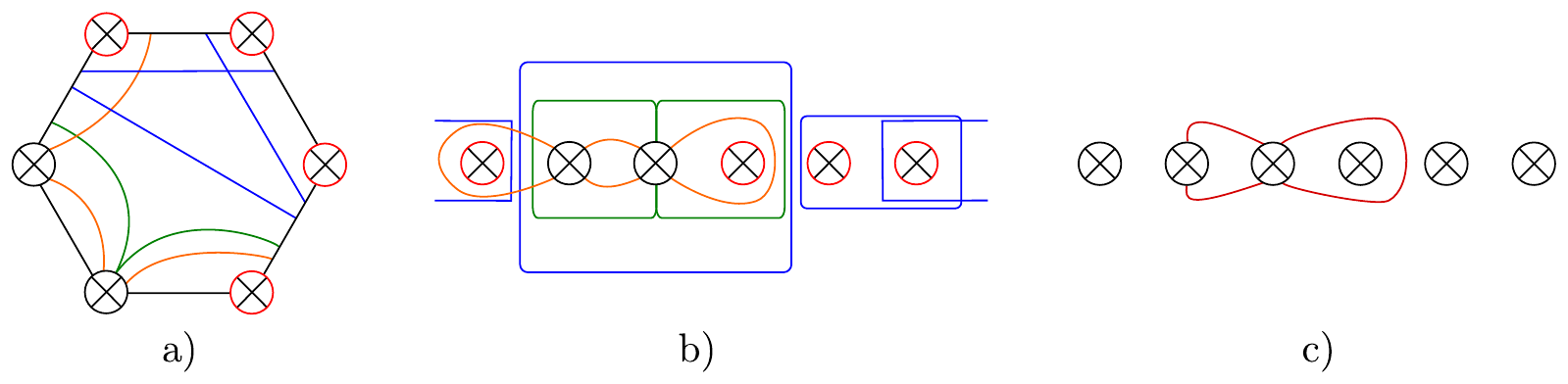}\caption{(a) and (b) depict the curves (from two angles) necessary to uniquely determine $\dt{\gamma_{i}}(\alpha_{i+1,i+2})$, which is depicted in (c).}\label{fig:lemma-dtgammaiAp1-alphai+1i+2}
 \end{figure}

For a fixed $1 \leq i \leq g$, we define the following auxiliary curve: $$E = (\E \setminus \{\veps_{i-3,i-1}, \veps_{i-2,i}, \veps_{i-1,i+1}, \veps_{i,i+2}\}) \cup \{\veps_{i-2,i+1}\} \subset V(\Y^{1}),$$ $$A_{1} = \{\alpha_{k} : k \neq i,i+1\} \subset V(\Y^{1}),$$ $$A_{2} = \{\alpha_{i,i-2}, \alpha_{i,i+1}\} \subset V(\Y),$$ $$M = \{\mu^{2}_{i}\} \subset V(\Y^{4}).$$ Then we define $\vkap^{1}_{i} := \langle E \cup A_{1} \cup A_{2} \cup M \rangle \in V(\Y^{5})$. Note that $\vkap^{1}_{i} = \dt{\gamma_{i}}^{-1}(\alpha_{i,i-2})$.

On the other hand, we now verify that for all $1 \leq i \leq g$, $\gamma_{i} \in V(\Y^{3})$.

\begin{lemma}\label{lemma:gammai}
 For all $1 \leq i \leq g$, we have that $\gamma_{i} \in V(\Y^{5})$.
\end{lemma}
\begin{proof}
 Let $1 \leq i \leq g$ be fixed. We define the following sets: $$E = \E \setminus \{\veps_{i-2,i}, \veps_{i,i+2}\} \subset V(\Y^{1}),$$ $$A = \{\alpha_{k}: k \neq i, i+1\} \subset V(\Y^{1}),$$ $$M = \{\mu^{2}_{i}\} \subset V(\Y^{4}).$$ Then $\gamma_{i} = \langle E \cup A \cup M \rangle \in V(\Y^{5})$. See Figure \ref{fig:lemma-gammai}.
\end{proof}

\begin{figure}[ht]
  \centering
  \includegraphics[height=3.5cm]{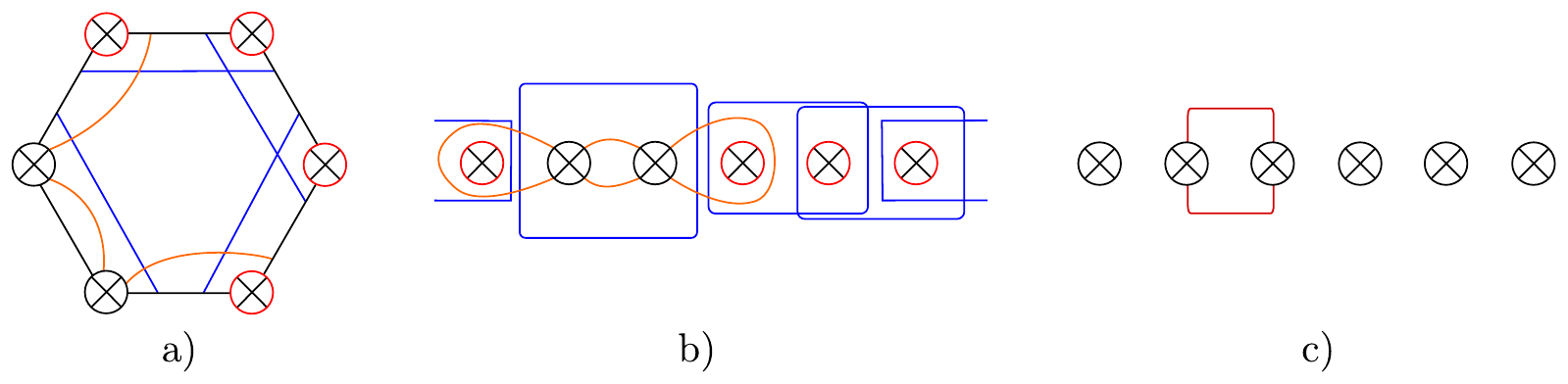}\caption{(a) and (b) depict the curves (from two angles) necessary to uniquely determine $\gamma_{i}$, which is depicted in (c).}\label{fig:lemma-gammai}
 \end{figure}

With this lemma we can now define the following auxiliary curve. Let $$C = \{\gamma_{i}: 1 \leq i \leq g\}.$$ Then we define $\eta := \langle C \rangle \in V(\Y^{6})$. See Figure \ref{fig:Def-eta}.

\begin{figure}[ht]
  \centering
  \includegraphics[height=3.5cm]{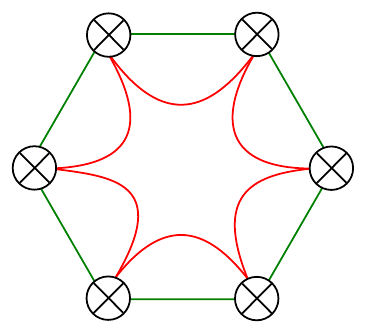}\caption{The curves in $C$ in red, and the curve $\eta = \langle C \rangle$ in green.}\label{fig:Def-eta}
 \end{figure}

\begin{lemma}\label{lemma:dtgammaiAp2}
 For all $1 \leq i \leq g-1$, we have that $\dt{\gamma_{i}}^{\pm 1}(\alpha_{i,i+1})$, $\dt{\gamma_{i}}^{\pm 1}(\alpha_{i+1,i-1}) \in V(\Y^{8})$.
\end{lemma}
\begin{proof}
 Let $1 \leq i \leq g-1$ be fixed. First we define an auxiliary curve: $$E = (\E \setminus \{\veps_{i-3,i-1}, \veps_{i-2,i}, \veps_{i-1,i+1}, \veps_{i,i+2}\}) \cup \{\veps_{i-2,i+1}\} \subset V(\Y^{1}),$$ $$A = \{\alpha_{k}: k \neq i-1, i, i+1\} \subset V(\Y^{1}),$$ $$K = \{\vkap^{1}_{i}, \eta\} \subset V(\Y^{6}).$$ With these sets we define $\gamma^{\prime}_{i} := \langle E \cup A \cup K \rangle \in V(\Y^{7})$. See Figure \ref{fig:lemma-dtgammaiAp2-gammaprime}.
 
 Now, to prove the lemma we define the following sets: $$E = \E \setminus \{\veps_{i-2,i}, \veps_{i,i+2}\} \subset V(\Y^{1}),$$ $$A = \{\alpha_{k}: k \neq i, i+1\} \subset V(\Y^{1}),$$ $$C = \{\gamma^{\prime}_{i}\} \subset V(\Y^{7}).$$ Then, we have that $\dt{\gamma_{i}}(\alpha_{i,i+1}) = \langle E \cup A \cup C \rangle \in V(\Y^{8})$. See Figure \ref{fig:lemma-dtgammaiAp2-alphaii+1}. Analogously, we have that $\dt{\gamma_{i}}^{-1}(\alpha_{i+1,i-1}) \in V(\Y^{8})$.
 
 Given that $\dt{\gamma_{i}}^{-1}(\alpha_{i,i+1}) = \alpha_{i+1}$, we have that $\dt{\gamma_{i}}^{-1}(\alpha_{i,i+1}) \in V(\Y^{1})$. Analogously, we have that $\dt{\gamma_{i}}(\alpha_{i+1,i-1}) \in V(\Y^{1})$.
\end{proof}

\begin{figure}[ht]
  \centering
  \includegraphics[height=3.5cm]{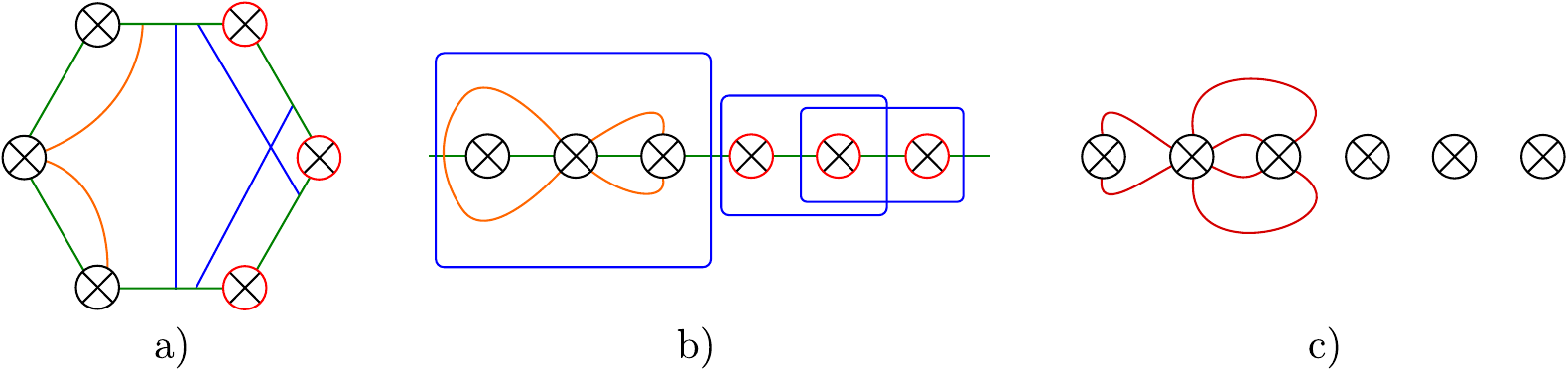}\caption{(a) and (b) depict the curves (from two angles) necessary to uniquely determine $\gamma_{i}^{\prime}$ (needed to uniquely determine $\dt{\gamma_{i}}(\alpha_{i,i+1})$), which is depicted in (c).}\label{fig:lemma-dtgammaiAp2-gammaprime}
 \end{figure}

\begin{figure}[ht]
  \centering
  \includegraphics[height=3.5cm]{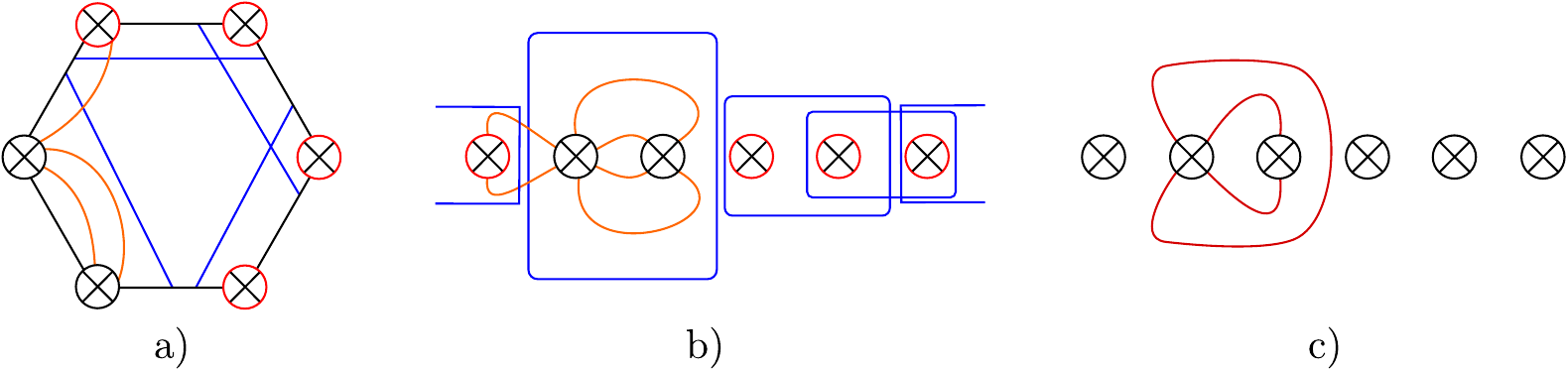}\caption{(a) and (b) depict the curves (from two angles) necessary to uniquely determine $\dt{\gamma_{i}}(\alpha_{i,i+1})$, which is depicted in (c).}\label{fig:lemma-dtgammaiAp2-alphaii+1}
 \end{figure}

\begin{lemma}\label{lemma:dtgammaiAp3}
 For all $1 \leq i \leq g-1$, we have that $\dt{\gamma_{i}}^{\pm 1}(\alpha_{i-1,i})$, $\dt{\gamma_{i}}^{\pm 1}(\alpha_{i+2,i}) \in V(\Y^{5})$.
\end{lemma}
\begin{proof}
 Let $1 \leq i \leq g-1$ be fixed. First we define an auxiliary curve for each $1 \leq j \leq g$: 
 $$E = (\E \setminus \{\veps_{j-3,j-1}, \veps_{j-2,j}, \veps_{j-1,j+1}, \veps_{j,j+2}, \veps_{j+1,j+3}\}) \cup \{\veps_{j-2,j+2}\} \subset V(\Y^{1}),$$
 $$A_{1} = \{\alpha_{k}: k \neq j,j+2\} \subset V(\Y^{1}),$$
 $$A_{2} = \{\alpha_{j,j-2},\alpha_{j,j+2}\} \subset V(\Y^{3}),$$
 $$M = \{\mu^{1}_{j}, \mu^{1}_{j+2}\} \subset V(\Y^{2}).$$
 With these sets we define $\gamma^{\prime}_{j}:= \langle E \cup A_{1} \cup A_{2} \cup M \rangle \in V(\Y^{4})$. See Figure \ref{fig:lemma-dtgammaiAp3-gammajprime}.
 
 Now, to prove the lemma we define the following sets: $$E = (\E \setminus \{\veps_{i-3,i-1}, \veps_{i-2,i}, \veps_{i-1,i+1}, \veps_{i,i+2}\}) \cup \{\veps_{i-2,i+1}\} \subset V(\Y^{1}),$$ $$A_{1} = \{\alpha_{k}: k \neq i-1, i+1\} \subset V(\Y^{1}),$$ $$A_{2} = \{\alpha_{i+1,i-1}\} \subset V(\Y),$$ $$M = \{\mu^{1}_{i+1}\} \cup \{\gamma^{\prime}_{i-1}\} \subset V(\Y^{4}).$$ Then, we have that $\dt{\gamma_{i}}(\alpha_{i-1,i}) = \langle E \cup A_{1} \cup A_{2} \cup M \rangle \subset V(\Y^{5})$. See Figure \ref{fig:lemma-dtgammaiAp3-alphai-1i}. Analogously, we have that $\dt{\gamma_{i}}^{-1}(\alpha_{i+2,i}) \in V(\Y^{5})$.
 
 Given that $\dt{\gamma_{i}}^{-1}(\alpha_{i-1,i}) = \dt{\gamma_{i-1}}(\alpha_{i,i+1})$, we have that $\dt{\gamma_{i}}^{-1}(\alpha_{i-1,i}) \in V(\Y^{5})$. Analogously, we have that $\dt{\gamma_{i}}(\alpha_{i+2,i}) \in V(\Y^{5})$.
\end{proof}

Finally, by Lemmata \ref{lemma:dtgammaiAp1}, \ref{lemma:dtgammaiAp2} and \ref{lemma:dtgammaiAp3}, and the observation at the beginning of this subsubsection, we have that $\dt{\gamma_{i}}^{\pm 1}(\A) \subset V(\Y^{8})$, finishing the proof of Lemma \ref{lemma:dtgammaiA}.

\begin{figure}[ht]
  \centering
  \includegraphics[height=3.5cm]{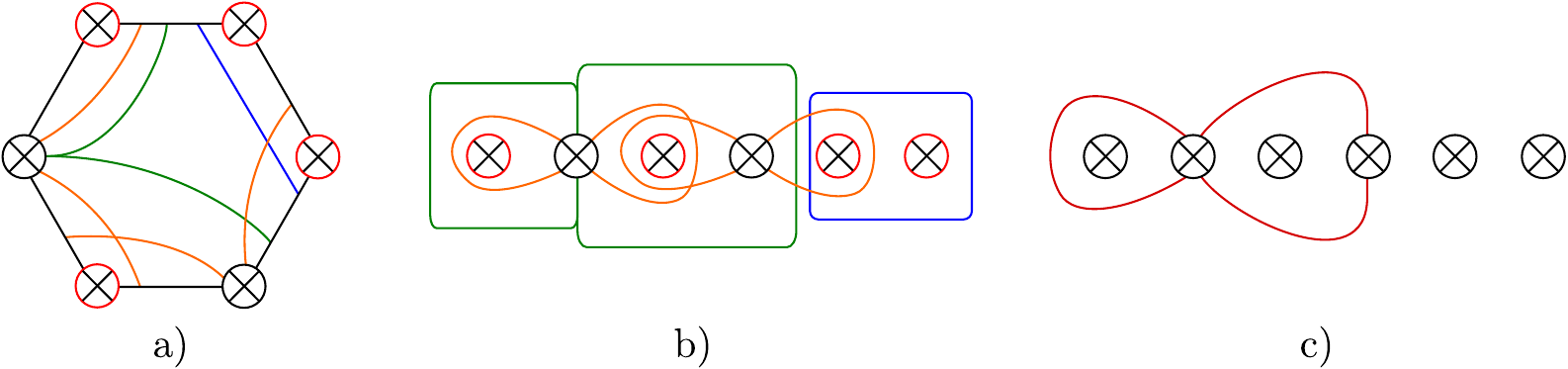}\caption{(a) and (b) depict the curves (from two angles) necessary to uniquely determine $\gamma_{j}^{\prime}$ (needed to uniquely determine $\dt{\gamma_{i}}(\alpha_{i-1,i})$), which is depicted in (c).}\label{fig:lemma-dtgammaiAp3-gammajprime}
 \end{figure}

\begin{figure}[ht]
  \centering
  \includegraphics[height=3.5cm]{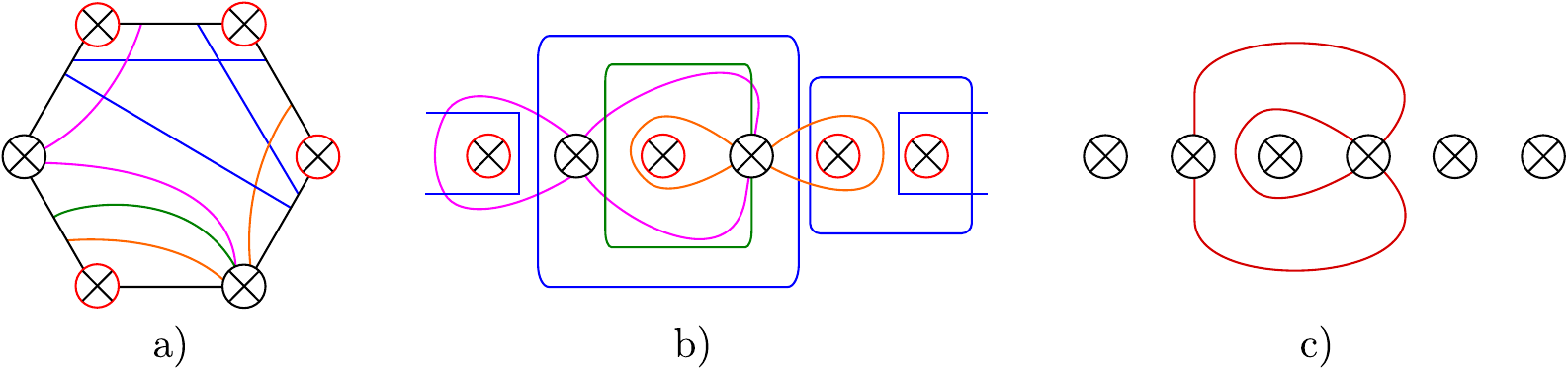}\caption{(a) and (b) depict the curves (from two angles) necessary to uniquely determine $\dt{\gamma_{i}}(\alpha_{i-1,i})$), which is depicted in (c).}\label{fig:lemma-dtgammaiAp3-alphai-1i}
 \end{figure}

\subsubsection{Proof of Lemma \ref{lemma:dtgammaiM}}\label{subsubsec:dtgammaiM}

Given that $\dt{\gamma_{i}}^{\pm 1}(\nu^{-}_{5,7}) = \nu^{-}_{5,7}$ if $i(\gamma_{i},\nu^{-}_{5,7}) = 0$, to prove Lemma \ref{lemma:dtgammaiM} we only need to verify that $\dt{\gamma_{i}}^{\pm 1}(\nu^{-}_{5,7}) \in V(\Y^{11})$ for $i = 4, 5, 6, 7$ (modulo $g$).

Before proving that, we need the following results.

\begin{corollary}\label{cor:dtgammaialphaivepsij}
 For all $1 \leq i \leq g-1$ and for all $1 \leq j,k \leq g$, we have that $\dt{\gamma_{i}}^{\pm 1}(\alpha_{j})$, $\dt{\gamma_{i}}^{\pm 1}(\veps_{j,k}) \in V(\Y^{9})$.
\end{corollary}
\begin{proof}
 By Lemmata \ref{lemma:vepsijv2} and \ref{lemma:alphai}, there exist subsets $A_{1}, A_{2} \subset \A$, such that $\alpha_{j} = \langle A_{1} \rangle$ and $\veps_{j,k} = \langle A_{2} \rangle$. Thus $\dt{\gamma_{i}}^{\pm 1}(\alpha_{j}) = \langle \dt{\gamma_{i}}^{\pm 1}(A_{1}) \rangle$ and $\dt{\gamma_{i}}^{\pm 1}(\veps_{i,j}) = \langle \dt{\gamma_{i}}^{\pm 1}(A_{2}) \rangle$.
 
 By Lemma \ref{lemma:dtgammaiA} we then have that $\dt{\gamma_{i}}^{\pm 1}(\alpha_{j})$, $\dt{\gamma_{i}}^{\pm 1}(\veps_{i,j}) \in V(\Y^{9})$ as desired.
\end{proof}

\begin{lemma}\label{lemma:dtgammainujkp1}
 Let $1 \leq i \leq g-1$ be fixed, then we have that $\dt{\gamma_{i}}^{\pm 1}(\nu^{\pm}_{i+1,i+3})$, $\dt{\gamma_{i+3}}^{\pm 1}(\nu^{\pm}_{i+1,i+3}) \in V(\Y^{10})$.
\end{lemma}
\begin{proof}
 We define the following sets: $$A = \{\alpha_{k}: 1 \leq k \leq g\} \subset V(\Y^{1}),$$ $$E = (\E \setminus \{\veps_{i-1,i+1}, \veps_{i,i+2}, \veps_{i+1,i+3}, \veps_{i+2,i+4}\}) \cup \{\veps_{i,i+3}\}\subset V(\Y^{1}),$$ $$M^{+}_{1} = \{\nu^{+}_{i+2,i+4}\} \subset V(\Y^{6}),$$ $$M^{-}_{1} = \{\nu^{-}_{i+2,i+4}\} \subset V(\Y^{6}),$$ $$M^{+}_{2} = \{\nu^{+}_{i,i+2}\} \subset V(\Y^{6}),$$ $$M^{-}_{2} = \{\nu^{-}_{i,i+2}\} \subset V(\Y^{6}).$$ Then, we have that: $$\nu^{-}_{i+1,i+3} = \langle A \cup E \cup M^{+}_{1} \rangle = \langle A \cup E \cup M^{+}_{2}\rangle,$$ $$\nu^{+}_{i+1,i+3} = \langle A \cup E \cup M^{-}_{1} \rangle = \langle A \cup E \cup M^{-}_{2}\rangle.$$ See Figure \ref{fig:def-numi+1i+3}.
 
 This implies that $$\dt{\gamma_{i}}^{\pm 1}(\nu^{-}_{i+1,i+3}) = \langle \dt{\gamma_{i}}^{\pm 1}(A) \cup \dt{\gamma_{i}}^{\pm 1}(E) \cup \dt{\gamma_{i}}^{\pm 1}(M^{+}_{1}) \rangle,$$
 $$\dt{\gamma_{i+3}}^{\pm 1}(\nu^{-}_{i+1,i+3}) = \langle \dt{\gamma_{i+3}}^{\pm 1}(A) \cup \dt{\gamma_{i+3}}^{\pm 1}(E) \cup \dt{\gamma_{i+3}}^{\pm 1}(M^{+}_{2})\rangle,$$
 $$\dt{\gamma_{i}}^{\pm 1}(\nu^{+}_{i+1,i+3}) = \langle \dt{\gamma_{i}}^{\pm 1}(A) \cup \dt{\gamma_{i}}^{\pm 1}(E) \cup \dt{\gamma_{i}}^{\pm 1}(M^{-}_{1}) \rangle,$$
 $$\dt{\gamma_{i+3}}^{\pm 1}(\nu^{+}_{i+1,i+3}) = \langle \dt{\gamma_{i+3}}^{\pm 1}(A) \cup \dt{\gamma_{i+3}}^{\pm 1}(E) \cup \dt{\gamma_{i+3}}^{\pm 1}(M^{-}_{2})\rangle.$$
 
 By Corollary \ref{cor:dtgammaialphaivepsij}, we have that $\dt{\gamma_{i}}^{\pm 1}(A)$, $\dt{\gamma_{i+3}}^{\pm 1}(A)$, $\dt{\gamma_{i}}^{\pm 1}(E)$, $\dt{\gamma_{i+3}}^{\pm 1}(E) \subset V(\Y^{9})$. Given that $i(\gamma_{i}, \nu^{\pm}_{i+2,i+4}) = i(\gamma_{i+3}, \nu^{\pm}_{i,i+2}) = 0$, we have that $\dt{\gamma_{i}}^{\pm 1}(M^{\pm}_{1}) = M^{\pm}_{1} \subset V(\Y^{6})$, and $\dt{\gamma_{i+3}}^{\pm 1}(M^{\pm}_{2}) = M^{\pm}_{2} \subset V(\Y^{6})$, we can conclude that $\dt{\gamma_{i}}^{\pm 1}(\nu^{\pm}_{i+1,i+3})$, $\dt{\gamma_{i+3}}^{\pm 1}(\nu^{\pm}_{i+1,i+3}) \in V(\Y^{10})$ as desired.
\end{proof}

\begin{figure}[ht]
  \centering
  \includegraphics[height=3.5cm]{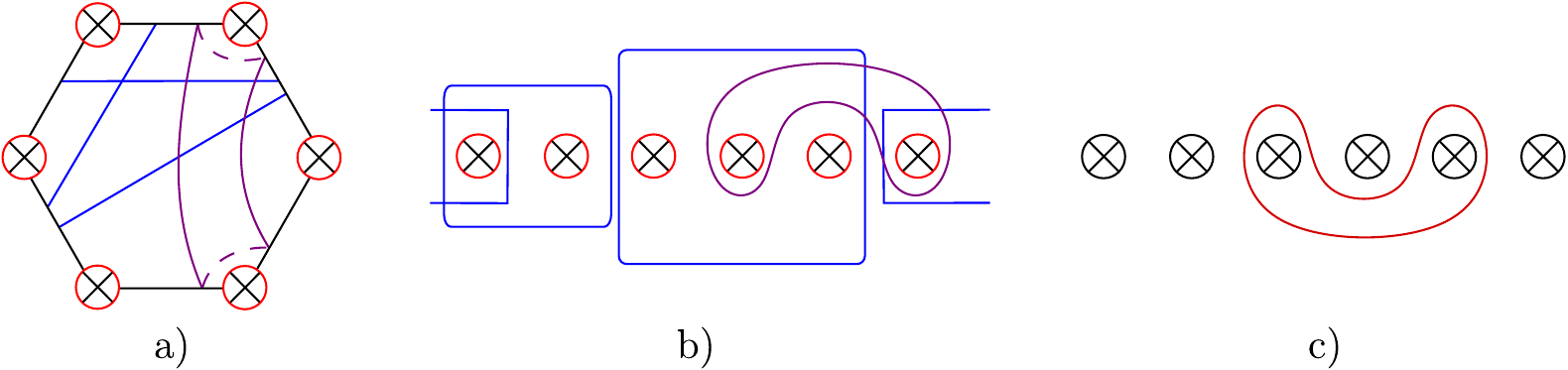}\caption{(a) and (b) depict the curves (from two angles) necessary to uniquely determine $\nu_{i+1,i+3}^{-}$, which is depicted in (c).}\label{fig:def-numi+1i+3}
 \end{figure}

\begin{corollary}\label{cor:dtgammainujkp2}
 Let $1 \leq i \leq g-1$ be fixed, then we have that $\dt{\gamma_{i+1}}^{\pm 1}(\nu^{\pm}_{i+1,i+3})$, $\dt{\gamma_{i+2}}^{\pm 1}(\nu^{\pm}_{i+1,i+3}) \in V(\Y^{11})$.
\end{corollary}
\begin{proof}
 Let $A$, $E$, $M^{\pm}_{1}$ and $M^{\pm}_{2}$ be as in the proof of Lemma \ref{lemma:dtgammainujkp1}. Then, as in the proof of Lemma \ref{lemma:dtgammainujkp1}, we have that: 
 $$\dt{\gamma_{i+1}}^{\pm 1}(\nu^{-}_{i+1,i+3}) = \langle \dt{\gamma_{i+1}}^{\pm 1}(A) \cup \dt{\gamma_{i+1}}^{\pm 1}(E) \cup \dt{\gamma_{i+1}}^{\pm 1}(M^{+}_{1}) \rangle,$$
 $$\dt{\gamma_{i+2}}^{\pm 1}(\nu^{-}_{i+1,i+3}) = \langle \dt{\gamma_{i+2}}^{\pm 1}(A) \cup \dt{\gamma_{i+2}}^{\pm 1}(E) \cup \dt{\gamma_{i+2}}^{\pm 1}(M^{+}_{2})\rangle,$$
 $$\dt{\gamma_{i+1}}^{\pm 1}(\nu^{+}_{i+1,i+3}) = \langle \dt{\gamma_{i+1}}^{\pm 1}(A) \cup \dt{\gamma_{i+1}}^{\pm 1}(E) \cup \dt{\gamma_{i+1}}^{\pm 1}(M^{-}_{1}) \rangle,$$
 $$\dt{\gamma_{i+2}}^{\pm 1}(\nu^{+}_{i+1,i+3}) = \langle \dt{\gamma_{i+2}}^{\pm 1}(A) \cup \dt{\gamma_{i+2}}^{\pm 1}(E) \cup \dt{\gamma_{i+2}}^{\pm 1}(M^{-}_{2})\rangle.$$
 
 Again, by Corollary \ref{cor:dtgammaialphaivepsij} we have that $\dt{\gamma_{i+1}}^{\pm 1}(A)$, $\dt{\gamma_{i+2}}^{\pm 1}(A)$, $\dt{\gamma_{i+1}}^{\pm 1}(E)$, $\dt{\gamma_{i+2}}^{\pm 1}(E) \subset V(\Y^{9})$. Now, by Lemma \ref{lemma:dtgammainujkp1} we have that $\dt{\gamma_{i+1}}^{\pm 1}(M^{\pm}_{1}), \dt{\gamma_{i+2}}^{\pm 1}(M^{\pm}_{2}) \subset V(\Y^{10})$. Therefore, $\dt{\gamma_{i+1}}^{\pm 1}(\nu^{\pm}_{i+1,i+3})$, $\dt{\gamma_{i+2}}^{\pm 1}(\nu^{\pm}_{i+1,i+3}) \in V(\Y^{11})$ as desired.
\end{proof}

Lemma \ref{lemma:dtgammaiM} follows directly from Lemma \ref{lemma:dtgammainujkp1} and Corollary \ref{cor:dtgammainujkp2}.


\subsection{$y^{\pm 1}(\Y)\subset \Y^{13}$}\label{subsec:y(Y)}

In this subsection we prove the following lemma.
\begin{lemma}\label{lemma:y(Y)}
We have that $y^{\pm 1}(\A)$, $y^{\pm 1}(\M) \subset \Y^{13}$.
\end{lemma}

The proof of Lemma \ref{lemma:y(Y)} needs several auxiliary curves and lemmata. We begin with two corollaries of Lemma \ref{lemma:dtgammaiA}  and Corollary \ref{cor:dtgammaialphaivepsij}.

\begin{corollary}\label{coro:dtgammaimuj1}
For all $1 \leq i \leq g$ and for all $1\leq j \leq g$ we have that $\dt{\gamma_i}^{\pm}(\mu_{j}^{1})\in V(\Y^{10})$.
\end{corollary}
\begin{proof}
Let $E_1$, $E_2$, $A_1$, $A_2$ as in the definition of $\mu_{j}^{1}$ (see Subsubsection \ref{subsubsec:mui}). We have that:
$$\dt{\gamma_{i}}^{\pm 1}(\mu_{j}^{1}) = \langle \dt{\gamma_{i}}^{\pm 1}(E_1) \cup \dt{\gamma_{i}}^{\pm 1}(E_2) \cup \gamma_{i}^{\pm 1}(A_1)\cup \dt{\gamma_{i}}^{\pm 1}(A_2) \rangle.$$
By Lemma \ref{lemma:dtgammaiA} we have that $\dt{\gamma_i}^{\pm}(A_1)\subset V(\Y^{8})$ and by Corollary \ref{cor:dtgammaialphaivepsij} we have that $\dt{\gamma_{i}}^{\pm 1}(E_1)$, $\dt{\gamma_{i}}^{\pm 1}(E_2)$, $\dt{\gamma_{i}}^{\pm 1}(A_2) \subset V(\Y^{9})$. Therefore $\dt{\gamma_{i}}^{\pm 1}(\mu_{j}^{1})\in V(\Y^{10})$.
\end{proof}

\begin{corollary}\label{coro:dtgammaialphajk}
For all $1\leq i \leq g$ and for all $1\leq j,k \leq g$ we have that $\dt{\gamma_i}^{\pm}(\alpha_{j,k})\in V(\Y^{11})$.
\end{corollary}
\begin{proof}
Let $1 \leq i \leq g$ fixed and $1\leq j,k \leq g$. If $k\in\{j-2,j-1,j,j+1 \}$ then $\alpha_{j,k}$ is either $\alpha_j$ or is an element of $\A$, then by Lemma \ref{lemma:dtgammaiA} and Corollary \ref{cor:dtgammaialphaivepsij} we have that  $\dt{\gamma_i}^{\pm}(\alpha_{j,k})\in V(\Y^{8})$ or  $\dt{\gamma_i}^{\pm}(\alpha_{j,k})\in V(\Y^{9})$. Otherwise let $E_1$, $E_2$, $E_3$, $A$ and $M$ as in the proof on  Lemma \ref{lemma:alphaij}. We have that:
$$\dt{\gamma_{i}}^{\pm 1}(\alpha_{j,k}) = \langle \dt{\gamma_{i}}^{\pm 1}(E_1) \cup \dt{\gamma_{i}}^{\pm 1}(E_2) \cup \dt{\gamma_{i}}^{\pm 1}(E_3) \cup \gamma_{i}^{\pm 1}(A)\cup \dt{\gamma_{i}}^{\pm 1}(M) \rangle.$$
By Corollary \ref{cor:dtgammaialphaivepsij} we have that $\dt{\gamma_{i}}^{\pm 1}(E_1)$, $\dt{\gamma_{i}}^{\pm 1}(E_2)$, $\dt{\gamma_{i}}^{\pm 1}(E_3)$, $\dt{\gamma_{i}}^{\pm 1}(A) \subset V(\Y^{8})$ and by Corollary \ref{coro:dtgammaimuj1} we have that $\dt{\gamma_{i}}^{\pm 1}(M)\in V(\Y^{10})$. Therefore $\dt{\gamma_{i}}^{\pm 1}(\alpha_{j,k})\in V(\Y^{11})$.
\end{proof}

For $1 \leq i \leq g$ we define $\mu_{i,i+2}^{\pm}$ and $\mu_{i,i-2}^{\pm}$ as follows (the subindices are modulo $g$). Define the sets:
$$ A = \{\alpha_k : k\neq i \} \subset V(\Y^1),$$
$$E_1 = (\E \setminus \{\veps_{i-3,i-1}, \veps_{i-2,i}, \veps_{i-1,i+1}, \veps_{i,i+2}, \veps_{i+1,i+3}\}) \cup \{\veps_{i-2,i+2}\}\subset V(\Y^{1}),$$
$$E_2 = (\E \setminus \{\veps_{i-4,i-2}, \veps_{i-3,i-1}, \veps_{i-2,i}, \veps_{i-1,i+1}, \veps_{i,i+2}\}) \cup \{\veps_{i-3,i+1}\}\subset V(\Y^{1}),$$
$$A_{1} = \{\alpha_{i,i-2},\alpha_{i,i-3}\}\subset V(\Y^{3}), \qquad A_{2} = \{\alpha_{i,i+1},\alpha_{i,i+2}\}\subset V(\Y^{3}),$$
$$M_{1}^{+} = \{\nu_{i+1,i+3}^{+}\} \subset V(\Y^{6}), \qquad M_{2}^{+} = \{\nu_{i-3,i-1}^{+}\} \subset V(\Y^{6}), $$
$$M_{1}^{-} = \{\nu_{i+1,i+3}^{-}\}\subset V(\Y^{6}), \qquad M_{2}^{-} = \{\nu_{i-3,i-1}^{-}\}\subset V(\Y^{6}).$$
Then
$$\mu_{i,i+2}^{+} := \langle A \cup E_{1} \cup A_{1} \cup M_{1}^{-}\rangle \subset V(\Y^{7}),$$
$$\mu_{i,i+2}^{-} := \langle A \cup E_{1} \cup A_{1} \cup M_{1}^{+} \rangle \subset V(\Y^{7}),$$
$$\mu_{i,i-2}^{+} := \langle A \cup E_{2} \cup A_{2} \cup M_{2}^{-}\rangle \subset V(\Y^{7}),$$
$$\mu_{i,i-2}^{-} := \langle A \cup E_{2} \cup A_{2} \cup M_{2}^{+} \rangle \subset V(\Y^{7}).$$
See Figure \ref{fig:Def-muii+2pm}.

\begin{figure}[ht]
  \centering
  \includegraphics[height=3.5cm]{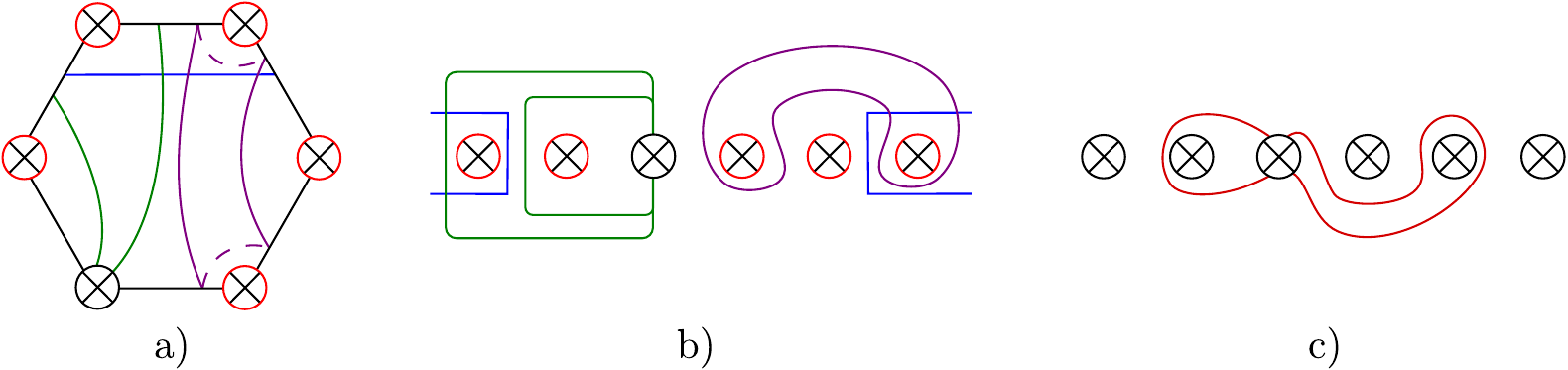}\caption{In (a) and (b) the curves of the set $A$ are in red, those of the set $E_1$ in blue, those of the set $A_1$ in green, and the curve of the set $M_{1}^{+}$ in purple. These curves uniquely determine the curve $\mu_{3,5}^{-}$ shown in (c).}\label{fig:Def-muii+2pm}
 \end{figure}

\begin{lemma}\label{lemma:dtgammaimupmiip2m2}
For $1\leq i \leq g$ fixed, we have that $\dt{\gamma_i}(\mu_{i,i+2}^{\pm})$, $\dt{\gamma_{i+1}}(\mu_{i,i+2}^{\pm})$, $\dt{\gamma_{i-1}}(\mu_{i,i-2}^{\pm})$, $\dt{\gamma_{i-2}}(\mu_{i,i-2}^{\pm}) \in V(\Y^{12}).$
\end{lemma}
\begin{proof}
Let $1\leq i \leq g$ fixed. There exist subsets $A$, $E_1$, $A_1$, $M_{1}^{+}$, $M_{1}^{-}\subset V(\Y^{6})$ such that  
$$\mu_{i,i+2}^{+} = \langle A\cup E_1\cup A_1\cup M_{1}^{-}\rangle,$$  
$$\mu_{i,i+2}^{-} = \langle A\cup E_1\cup A_1\cup M_{1}^{+}\rangle.$$
We have that $\dt{\gamma_i}(E_1) = \dt{\gamma_{i+1}}(E_1)= E_1\subset V(\Y^1)$ and that $\dt{\gamma_{i+1}}(A_1)=A_1 \subset V(\Y^{3})$. By Corollary \ref{coro:dtgammaialphajk} we have that $\dt{\gamma_i}(A_1)\subset V(\Y^{11})$. By Corollary \ref{cor:dtgammaialphaivepsij} we have that $\dt{\gamma_i}(A)$, $\dt{\gamma_{i+1}}(A)\subset V(\Y^{9})$. Finally by Lemma \ref{lemma:dtgammainujkp1} and Corollary \ref{cor:dtgammainujkp2} we have that $\dt{\gamma_{i}}(M_{1}^{\pm})\subset V(\Y^{10})$ and $\dt{\gamma_{i+1}}(M_{1}^{\pm})\subset V(\Y^{11})$. Therefore $\dt{\gamma_{i}}(\mu_{i,i+2}^{\pm})\subset V(\Y^{12})$ and $\dt{\gamma_{i+1}}(\mu_{i,i+2}^{\pm}) \subset V(\Y^{12})$.

Analogously, there exist subsets $A$, $E_2$, $A_2$, $M_{2}^{+}$, $M_{2}^{-}\subset V(\Y^{6})$ such that  
$$\mu_{i,i-2}^{+} = \langle A\cup E_2\cup A_2\cup M_{2}^{-}\rangle,$$  
$$\mu_{i,i-2}^{-} = \langle A\cup E_2\cup A_2\cup M_{2}^{+}\rangle.$$
We have that $\dt{\gamma_{i-1}}(E_2) = \dt{\gamma_{i-2}}(E_2)= E_2 \subset V(\Y^1)$ and that $\dt{\gamma_{i-2}}(A_2)=A_2 \subset V(\Y^{3})$. By Corollary \ref{coro:dtgammaialphajk} we have that $\dt{\gamma_{i-1}}(A_2)\subset V(\Y^{11})$. By Corollary \ref{cor:dtgammaialphaivepsij} we have that $\dt{\gamma_{i-1}}(A)$, $\dt{\gamma_{i-2}}(A)\subset V(\Y^{9})$. Finally by Lemma \ref{lemma:dtgammainujkp1} and Corollary \ref{cor:dtgammainujkp2} we have that $\dt{\gamma_{i-1}}(M_{2}^{\pm})\subset V(\Y^{10})$ and $\dt{\gamma_{i-2}}(M_{2}^{\pm})\subset V(\Y^{11})$. Therefore $\dt{\gamma_{i-1}}(\mu_{i,i-2}^{\pm})\subset V(\Y^{12})$ and $\dt{\gamma_{i-2}}(\mu_{i,i-2}^{\pm}) \subset V(\Y^{12})$.

\end{proof}

For $1 \leq i \leq g$ we define $\alpha_{i,i+2}^{\pm}$ and $\alpha_{i,i-2}^{\pm 1}$ as follows (the subindices are modulo $g$). Define the sets:
$$ A = \{\alpha_k : k\neq i \} \subset V(\Y^1),$$
$$M^{+} = \{\nu_{i-1,i+1}^{+}\} \subset V(\Y^{6}),$$
$$M^{-} = \{\nu_{i-1,i+1}^{-}\}\subset V(\Y^{6}),$$
$$E_1 = (\E \setminus \{\veps_{i-2,i}, \veps_{i-1,i+1}, \veps_{i,i+2}, \veps_{i+1,i+3}\}) \cup \{\veps_{i-1,i+2}\}\subset V(\Y^{1}),$$
$$E_2 = (\E \setminus \{\veps_{i-4,i-2}, \veps_{i-3,i-1}, \veps_{i-2,i}, \veps_{i-1,i+1}\}) \cup \{\veps_{i-3,i}\}\subset V(\Y^{1}),$$
$$C_{1}^{+} = \{\mu_{i,i+2}^{+}\} \subset V(\Y^{7}), \qquad C_{2}^{+} = \{\mu_{i,i-2}^{+}\} \subset V(\Y^{7}),$$
$$C_{1}^{-} = \{\mu_{i,i+2}^{-}\} \subset V(\Y^{7}), \qquad C_{2}^{-} = \{\mu_{i,i-2}^{-}\} \subset V(\Y^{7}).$$
Then
$$\alpha_{i,i+2}^{+} := \langle A \cup E_{1} \cup M^{-} \cup C_{1}^{+}\rangle \subset V(\Y^{8}),$$
$$\alpha_{i,i+2}^{-} := \langle A \cup E_{1} \cup M^{+} \cup C_{1}^{-}\rangle \subset V(\Y^{8}),$$
$$\alpha_{i,i-2}^{+} := \langle A \cup E_{2} \cup M^{-} \cup C_{2}^{+}\rangle \subset V(\Y^{8}),$$
$$\alpha_{i,i-2}^{-} := \langle A \cup E_{2} \cup M^{+} \cup C_{2}^{-}\rangle \subset V(\Y^{8}).$$
See Figure \ref{fig:Def-alphaii+2pm}.

\begin{figure}[ht]
  \centering
  \includegraphics[height=3.5cm]{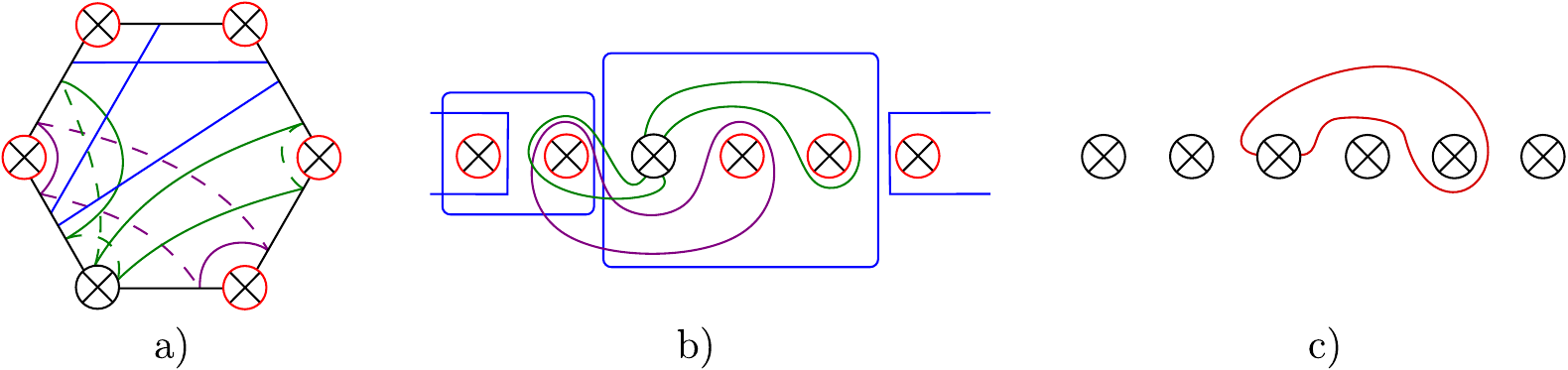}\caption{In (a) and (b) the curves of the set $A$ are in red, those of the set $E_1$ in blue, the curve of the set $M^{-}$ in purple, and the curve of the set $C_{1}^{+}$ in green. These curves uniquely determine the curve $\alpha_{3,5}^{+}$ shown in (c).}\label{fig:Def-alphaii+2pm}
 \end{figure}

\begin{lemma}\label{lemma:dtgammaialphapmiip2m2}
For $1\leq i \leq g$ fixed, we have that $\dt{\gamma_i}(\alpha_{i,i+2}^{\pm})$, $\dt{\gamma_{i+1}}(\alpha_{i,i+2}^{\pm})$, $\dt{\gamma_{i-1}}(\alpha_{i,i-2}^{\pm})$, $\dt{\gamma_{i-2}}(\alpha_{i,i-2}^{\pm}) \in V(\Y^{13}).$
\end{lemma}
\begin{proof}
For $1 \leq i \leq g$ fixed, there exist subsets $A$, $M^+$, $M^-$, $E_1$, $E_2$, $C_1^+$, $C_1^-$, $C_2^+$, $C_2^-\subset V(\Y^7)$ such that: 
$$\alpha_{i,i+2}^{+} := \langle A \cup E_{1} \cup M^{-} \cup C_{1}^{+}\rangle \subset V(\Y^{8}),$$
$$\alpha_{i,i+2}^{-} := \langle A \cup E_{1} \cup M^{+} \cup C_{1}^{-}\rangle \subset V(\Y^{8}),$$
$$\alpha_{i,i-2}^{+} := \langle A \cup E_{2} \cup M^{-} \cup C_{2}^{+}\rangle \subset V(\Y^{8}),$$
$$\alpha_{i,i-2}^{-} := \langle A \cup E_{2} \cup M^{+} \cup C_{2}^{-}\rangle \subset V(\Y^{8}).$$
We have that $\dt{\gamma_i}(E_1) = \dt{\gamma_{i+1}}(E_1) = E_1\subset(\Y^1)$ and $\dt{\gamma_{i-1}}(E_2) = \dt{\gamma_{i-2}}(E_2) = E_2\subset(\Y^1)$. By Corollary \ref{cor:dtgammaialphaivepsij} we have that $\dt{\gamma_{i}}(A)$, $\dt{\gamma_{i+1}}(A)$, $\dt{\gamma_{i-1}}(A)$, $\dt{\gamma_{i-2}}(A)\subset V(\Y^9)$. By Lemma \ref{lemma:dtgammainujkp1} we have that $\dt{\gamma_{i+1}}(M^{\pm})$, $\dt{\gamma_{i-2}}(M^{\pm}) \subset V(\Y^{10})$. By Corollary \ref{cor:dtgammainujkp2} we have that $\dt{\gamma_{i}}(M^{\pm})$, $\dt{\gamma_{i-1}}(M^{\pm})\subset V(\Y^{11})$. Finally, by Lemma \ref{lemma:dtgammaimupmiip2m2} we have that $\dt{\gamma_{i}}(C_{1}^{\pm})$, $\dt{\gamma_{i+1}}(C_{1}^{\pm})$, $\dt{\gamma_{i-1}}(C_{2}^{\pm})$, $\dt{\gamma_{i-2}}(C_{2}^{\pm})\subset V(\Y^{12})$. Therefore $\dt{\gamma_i}(\alpha_{i,i+2}^{\pm})$, $\dt{\gamma_{i+1}}(\alpha_{i,i+2}^{\pm})$, $\dt{\gamma_{i-1}}(\alpha_{i,i-2}^{\pm})$, $\dt{\gamma_{i-2}}(\alpha_{i,i-2}^{\pm}) \in V(\Y^{13}).$
\end{proof}

\begin{lemma}\label{lemma:ygammaiip1m2}
For all $1\leq  i \leq g$ we have that $y^{\pm 1}(\alpha_{i, i+1})$, $y^{\pm 1}(\alpha_{i, i-2})\in  V(\Y^{13})$.
\end{lemma}
\begin{proof}
Given that $y^{\pm 1}(\alpha_{i,i+1})= \alpha_{i,i+1}$ if $1\leq i \leq g-3$, and that $y^{\pm 1}(\alpha_{i,i-2})= \alpha_{i,i-2}$ if $2\leq i \leq g-2$, we only need to verify that $$y^{\pm 1}(\alpha_{g-2, g-1}), y^{\pm 1}(\alpha_{g-1, g)}, y^{\pm 1}(\alpha_{g,1}), y^{\pm 1}(\alpha_{g-1, g-3}), y^{\pm 1}(\alpha_{g, g-2}), y^{\pm}(\alpha_{1, g-1})\in V(\Y^{13}).$$  

We that $y^{\pm 1}(\alpha_{g, g-2})= \alpha_g$ and that $y^{\pm 1}(\alpha_{g-1, g}) = \dt{\gamma_{g-1}}(\alpha_{g,g-1})$. Then, by Lemma \ref{lemma:alphai} we have that $y^{\pm 1}(\alpha_{g, g-2})\in V(\Y^1)$ and by Lemma \ref{lemma:dtgammaiA}  we have that $y^{\pm 1}(\alpha_{g-1, g})\in V(\Y^8)$.

Notice that
$$y^{\pm 1}(\alpha_{g-2, g-1}) = \dt{\gamma_{g-1}}(\alpha_{g-2,g}^{\pm}),$$
$$y^{\pm}(\alpha_{g,1}) = \dt{\gamma_{g-1}}(\alpha_{g-1,1}^{\pm}),$$
$$y^{\pm 1}(\alpha_{g-1, g-3}) = \dt{\gamma_{g-1}}(\alpha_{g,g-2}^{\pm}),$$
$$y^{\pm 1}(\alpha_{1, g-1}) = \dt{\gamma_{g-1}}(\alpha_{1,g-1}^{\pm}).$$
Then by Lemma \ref{lemma:dtgammaialphapmiip2m2} $y^{\pm 1}(\alpha_{g-2, g-1})$, $y^{\pm 1}(\alpha_{g,1})$, $y^{\pm 1}(\alpha_{g-1, g-3})$, $y^{\pm 1}(\alpha_{1, g-1})\in V(\Y^{13})$.
\end{proof}

We need more auxiliary curves. For $1 \leq i \leq g$ we define $\kappa_{i,i+3}^{\pm}$ as follows (the subindices are modulo $g$). Define the sets:
$$ A = \{\alpha_k : 1 \leq k \leq g\} \subset V(\Y^1),$$
$$E = (\E \setminus \{\veps_{i-2,i},\veps_{i-1,i+1}, \veps_{i,i+2}, \veps_{i+1,i+3}, \veps_{i+2,i+4} \}) \cup \{\veps_{i-1,i+3}\}\subset V(\Y^{1}),$$
$$M^{+} = \{\nu_{i-1,i+1}^{+}, \nu_{i+2,i+4}^{-}\} \subset V(\Y^{6}),$$
$$M^{-} = \{\nu_{i-1,i+1}^{-}, \nu_{i+2,i+4}^{+}\} \subset V(\Y^{6}).$$
We define 
$$\kappa_{i,i+3}^{+}:=\langle A\cup E\cup M^{-}\rangle,$$
$$\kappa_{i,i+3}^{-}:=\langle A\cup E\cup M^{+}\rangle.$$
See \Cref{fig:Def-kappaii+3pm}.

\begin{figure}[ht]
  \centering
  \includegraphics[height=3.5cm]{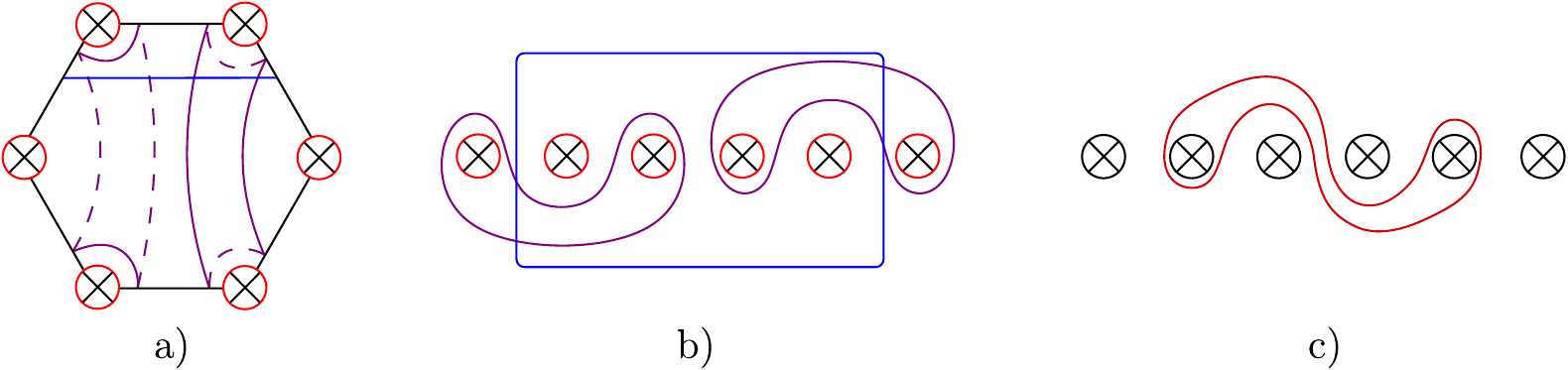}\caption{In (a) and (b) the curves of the set $A$ are in red, those of the set $E$ in blue, those of of the set $M^{-}$ in purple. These curves uniquely determine the curve $\kappa_{2,5}^{+}$ shown in (c).}\label{fig:Def-kappaii+3pm}
 \end{figure}

\begin{lemma}\label{lemma:dtgammaikappaiip3}
For $1\leq i \leq g $ fixed, we have that $\dt{\gamma_{i}}^{\pm 1}(\kappa_{i,i+3}^{\pm})$, $\dt{\gamma_{i+2}}^{\pm 1}(\kappa_{i,i+3}^{\pm}) \subset V(\Y^{12})$.
\end{lemma}
\begin{proof}
Let $A$, $E$, $M^{+}$, $M^{-}$ be as in the definition of $\kappa_{i,i+3}^{+}$ and $\kappa_{i,i+3}^{-}$, we have that:
$$\dt{\gamma_{i}}^{\pm 1}(\kappa_{i,i+3}^{+}) = \langle \dt{\gamma_{i}}^{\pm 1}(A)\cup \dt{\gamma_{i}}^{\pm 1}(E) \cup \dt{\gamma_{i}}^{\pm 1}(M^{-}) \rangle, $$
$$\dt{\gamma_{i}}^{\pm 1}(\kappa_{i,i+3}^{-}) = \langle \dt{\gamma_{i}}^{\pm 1}(A)\cup \dt{\gamma_{i}}^{\pm 1}(E) \cup \dt{\gamma_{i}}^{\pm 1}(M^{+}) \rangle, $$
$$\dt{\gamma_{i+2}}^{\pm 1}(\kappa_{i,i+3}^{+}) = \langle \dt{\gamma_{i+2}}^{\pm 1}(A)\cup \dt{\gamma_{i+2}}^{\pm 1}(E) \cup \dt{\gamma_{i+2}}^{\pm 1}(M^{-}) \rangle, $$
$$\dt{\gamma_{i+2}}^{\pm 1}(\kappa_{i,i+3}^{-}) = \langle \dt{\gamma_{i+2}}^{\pm 1}(A)\cup \dt{\gamma_{i+2}}^{\pm 1}(E) \cup \dt{\gamma_{i+2}}^{\pm 1}(M^{+}) \rangle.$$

By Corollary \ref{cor:dtgammaialphaivepsij}, we have that $\dt{\gamma_{i}}^{\pm 1}(A)$, $\dt{\gamma_{i+2}}^{\pm 1}(A)$, $\dt{\gamma_{i}}^{\pm 1}(E)$, $\dt{\gamma_{i+2}}^{\pm 1}(E) \in  V(\Y^{9})$. By  Corollary \ref{cor:dtgammainujkp2} $\dt{\gamma_{i}}^{\pm 1}(M^{\pm})$, $\dt{\gamma_{i+2}}^{\pm 1}(M^{\pm})\in V(\Y ^{11})$. Therefore $\dt{\gamma_{i}}^{\pm 1}(\kappa_{i,i+3}^{\pm})$, $\dt{\gamma_{i+2}}^{\pm 1}(\kappa_{i,i+3}^{\pm}) \subset V(\Y^{12})$.
\end{proof}

We need more auxiliary curves. For $1 \leq i \leq g$ we define  the curves $\eta_{i,i-1}^{\pm}$ and $\eta_{i,i+1}^{\pm}$ as follows (the subindices are modulo $g$). Define the sets:
$$ A = \{\alpha_k : 1 \leq k \leq g\} \subset V(\Y^1),$$
$$E = (\E \setminus \{\veps_{i-3,i-1}, \veps_{i-2,i},\veps_{i-1,i+1}, \veps_{i,i+2} \}) \cup \{\veps_{i-2,i+1}\}\subset V(\Y^{1}),$$
$$M_{1}^{+} = \{\kappa_{i-2,i+1}^{+} \} \subset V(\Y^{7}), \qquad M_{2}^{+} = \{\kappa_{i-1,i+2}^{+} \} \subset V(\Y^{7}),$$
$$M_{1}^{-} = \{\kappa_{i-2,i+1}^{-} \} \subset V(\Y^{7}), \qquad M_{2}^{-} = \{\kappa_{i-1,i+2}^{-} \} \subset V(\Y^{7}).$$
We define 
$$\eta_{i,i-1}^{+}:=\langle A\cup E\cup M_{1}^{-}\rangle,$$
$$\eta_{i,i-1}^{-}:=\langle A\cup E\cup M_{1}^{+}\rangle,$$
$$\eta_{i,i+1}^{+}:=\langle A\cup E\cup M_{2}^{+}\rangle,$$
$$\eta_{i,i+1}^{-}:=\langle A\cup E\cup M_{2}^{-}\rangle.$$
See Figure \ref{fig:Def-etaii-1pm}.

\begin{figure}[ht]
  \centering
  \includegraphics[height=3.5cm]{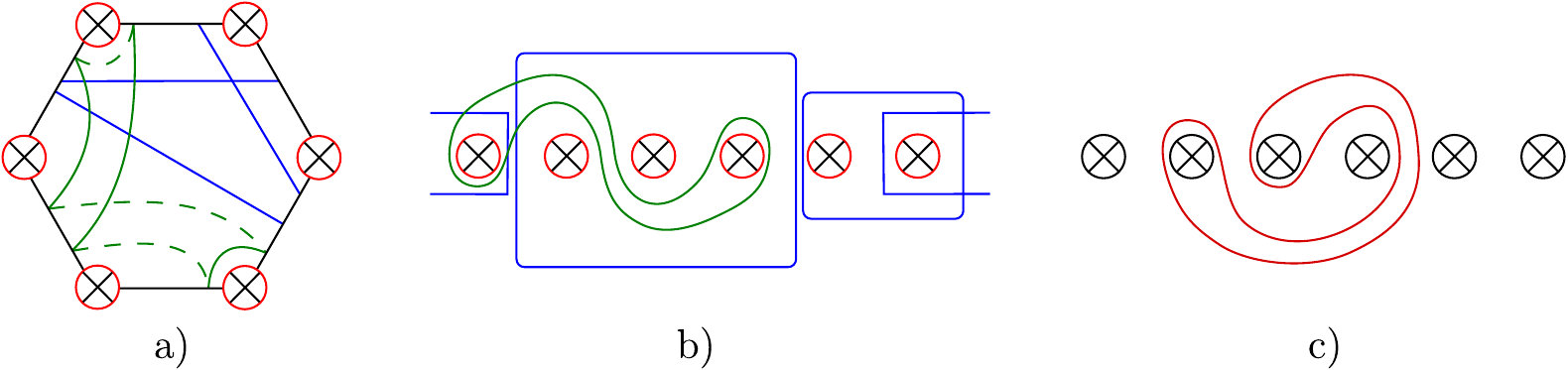}\caption{In (a) and (b) the curves of the set $A$ are in red, those of the set $E$ in blue, and the curve of the set $M_{1}^{+}$ in green. These curves uniquely determine the curve $\eta_{3,2}^{-}$ shown in (c).}\label{fig:Def-etaii-1pm}
 \end{figure}

\begin{lemma}\label{lemma:dtgammaietai}
For $1\leq i \leq g $ fixed, we have that $\dt{\gamma_{i-1}}^{\pm 1}(\eta_{i,i+1}^{\pm})$, $\dt{\gamma_{i}}^{\pm 1}(\eta_{i,i-1}^{\pm}) \subset V(\Y^{13})$.
\end{lemma}
\begin{proof}
Let $A$, $E$, $M_{1}^{+}$, $M_{1}^{-}$, $M_{2}^{+}$, $M_{2}^{-}$ be as in the definition of $\eta_{i,i+1}^{+}$ and $\eta_{i,i-1}^{-}$, we have that:
$$\dt{\gamma_{i-1}}^{\pm 1}(\eta_{i,i+1}^{+}) = \langle \dt{\gamma_{i-1}}^{\pm 1}(A)\cup \dt{\gamma_{i-1}}^{\pm 1}(E) \cup \dt{\gamma_{i-1}}^{\pm 1}(M_{2}^{+}) \rangle, $$
$$\dt{\gamma_{i-1}}^{\pm 1}(\eta_{i,i+1}^{+}) = \langle \dt{\gamma_{i-1}}^{\pm 1}(A)\cup \dt{\gamma_{i-1}}^{\pm 1}(E) \cup \dt{\gamma_{i-1}}^{\pm 1}(M_{2}^{-}) \rangle, $$
$$\dt{\gamma_{i}}^{\pm 1}(\eta_{i,i-1}^{+}) = \langle \dt{\gamma_{i}}^{\pm 1}(A)\cup \dt{\gamma_{i}}^{\pm 1}(E) \cup \dt{\gamma_{i}}^{\pm 1}(M_{1}^{-}) \rangle, $$
$$\dt{\gamma_{i}}^{\pm 1}(\eta_{i,i-1}^{-}) = \langle \dt{\gamma_{i}}^{\pm 1}(A)\cup \dt{\gamma_{i}}^{\pm 1}(E) \cup \dt{\gamma_{i}}^{\pm 1}(M_{1}^{+}) \rangle.$$

By Corollary \ref{cor:dtgammaialphaivepsij}, we have that $\dt{\gamma_{i-1}}^{\pm 1}(A)$, $\dt{\gamma_{i}}^{\pm 1}(A)$, $\dt{\gamma_{i-1}}^{\pm 1}(E)$, $\dt{\gamma_{i}}^{\pm 1}(E) \in  V(\Y^{9})$. By Lemma \ref{lemma:dtgammaikappaiip3} $\dt{\gamma_{i-1}}^{\pm 1}(M_{2}^{\pm})$, $\dt{\gamma_{i}}^{\pm 1}(M_{1}^{\pm})\in V(\Y ^{12})$. Therefore $\dt{\gamma_{i-1}}^{\pm 1}(\eta_{i,i+1}^{\pm})$, $\dt{\gamma_{i}}^{\pm 1}(\eta_{i,i-1}^{\pm}) \subset V(\Y^{13})$.
\end{proof}

\begin{lemma}\label{lemma:ypmmu57m}
For all $g\geq 6$ we have that $y^{\pm 1}(\nu_{5,7}^{-}) \in V(\Y^{13})$ (the subindices are modulo $g$).
\end{lemma}
\begin{proof}
If $g\geq 9 $, then $y^{\pm 1}(\nu_{5,7}^{-}) = \nu_{5,7}^{-} \in \Y$. Therefore we only need to verify  $y^{\pm 1}(\nu_{5,7}^{-}) \in V(\Y^{13})$ when $6\leq g \leq 8$. 
\begin{description}
\item[Case $g=8$:] We define the following sets:
$$A = \{ \alpha_{k} : k\neq 8\}\subset V(\Y^{1}),$$
$$E = (\E \setminus \{\veps_{3,5}, \veps_{4,6},\veps_{5,7}, \veps_{6,8}, \veps_{7,1} \}) \cup \{\veps_{4,8}\}\subset V(\Y^{1}),$$
$$M = \{\nu_{4,6}^{+} \}\subset V(\Y^{6}).$$
We have that $y(\nu_{5,7}^{-}) = \langle A \cup E \cup M \rangle\in V(\Y^{7})$. On the other hand, notice that $y^{-1}(\nu_{5,7}^{-}) = \dt{\gamma_7}(\kappa_{5,8}^{-})$, then by Lemma \ref{lemma:dtgammaikappaiip3} we have that $y^{-1}(\nu_{5,7}^{-})\in V(\Y^{12})$.
\item[Case $g=7$:]Notice that $y(\nu_{5,7}^{-}) = \dt{\gamma_6}(\eta_{6,5}^{-})$ and $y^{-1}(\nu_{5,7}^{-}) = \dt{\gamma_6}(\veps_{4,6})$. Therefore, by Lemma \ref{lemma:dtgammaietai}  $y(\nu_{5,7}^{-})\in V(\Y^{13})$ and by Corollary \ref{cor:dtgammaialphaivepsij} we have that $y^{-1}(\nu_{5,7}^{-}) \in V(\Y^{9})$.
\item[Case $g=6$:]Notice that $y(\nu_{5,1}^{-}) = \mu_{g}^{1}$ and $y^{-1}(\nu_{5,1}^{-}) = \dt{\gamma_5}(\eta_{6,1}^{-})$. By the definition of $\mu_{g}^{1}$ (see Subsubsection \ref{subsubsec:mui}) we have that $y(\nu_{5,1}^{-})\in V(\Y^{2})$ and by Lemma \ref{lemma:dtgammaietai} we have that $y^{-1}(\nu_{5,1}^{-})\in V(\Y^{13})$.
\end{description} 
\end{proof}

Lemma \ref{lemma:y(Y)} follows from Lemma \ref{lemma:ygammaiip1m2} and Lemma \ref{lemma:ypmmu57m}.


\subsection{$\tau_{\beta}^{\pm 1}(\Y)\subset \Y^{13}$}\label{subsec:y(Y)}

In this subsection we prove the following lemma.

\begin{lemma}\label{lemma:taubeta(Y)}
We have that $\tau_{\beta}^{\pm 1}(\Y) \subset \Y^{13}$.
\end{lemma}

Since $i(\delta, \beta) = 0$ implies that $\tau_{\beta}^{\pm 1}(\delta) = \delta$, to prove Lemma \ref{lemma:taubeta(Y)} we need only verify the curves in $\Y$ that intersect $\beta$. The set of such curves can be partitioned as follows:

$$C_{1} = \{\alpha_{1,2}, \alpha_{2,3}, \alpha_{3,4}, \alpha_{4,5}, \alpha_{1,g-1}, \alpha_{2,g}, \alpha_{3,1}, \alpha_{4,2}\},$$
$$C_{2} = \{\alpha_{5,3}, \alpha_{g,1}\},$$
$$C_{3} = \left\{\begin{array}{cl}
    \{\nu_{5,7}^{-}\} & \text{if $g = 6$,} \\
    \varnothing & \text{otherwise.}
\end{array}
\right.$$

Both the proof of Lemma \ref{lemma:taubeta(Y)} and this subsection is then partitioned as $\tau_{\beta}^{\pm 1}(C_1) \subset \Y^{10}$, $\tau_{\beta}^{\pm 1}(C_2) \subset \Y^{12}$ and $\tau_{\beta}^{\pm 1}(C_3) \subset \Y^{13}$.

\subsubsection{$\tau_{\beta}^{\pm 1}(C_1) \subset \Y^{10}$}\label{subsubsec:taubetaC1}

For this part of the proof, note first that for all $\delta \in C_1$ we have that $i(\delta, \beta) = 1$, and also note that both $\beta$ and the set $C_1$ are invariant under the reflection homeomorphism whose axis ``passes between'' the second and third genus (see Subsection \ref{subsec:model} and Figure \ref{Fig-reflection-betaC1}). As such, we need only prove that $\tau_{\beta}^{\pm 1} (\delta) \in \Y^{10}$ for half the curves (making sure that the curves used are either invariant under the same reflection or both the curves and their images are in the same rigid expansion). More specifically, we need only prove that $\tau_{\beta}^{\pm 1}(\{\alpha_{1,2}, \alpha_{2,3}, \alpha_{3,4}, \alpha_{4,5}\}) \subset \Y^{10}$.

\begin{figure}[ht]
  \centering
  \includegraphics[height=3.5cm]{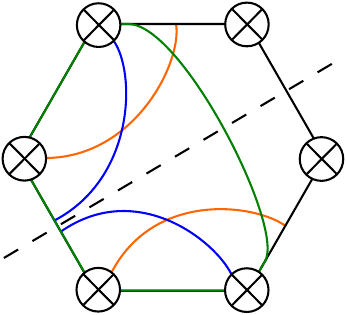}\caption{The curve $\beta$ in green, $\{\alpha_{1,2} , \, \alpha_{4,2}\}$ in blue and $\{\alpha_{3,4}, \, \alpha_{2,g}\}$ in orange.}\label{Fig-reflection-betaC1}
 \end{figure}

To do this, we introduce auxiliary curves that are used throughout this section.

Let $1 \leq i,j,k \leq g$ be such that $i$ and $j$ are not consecutive in the cyclic order in $\Z / g \Z$. See Table \ref{tab:auxcurvestaubeta} and Figure \ref{Fig-Def-aux-taubeta}.

\begin{table}[]
    $$\begin{array}{|c|c|c|}\hline
   \text{Auxiliary curve} & \text{Sets of curves whose union is needed} & k \text{th rigid expansion for the} \\
   & \text{to uniquely determine the auxiliary curve} & \text{auxiliary curve} \\ \hline
   \theta_{i,j}  & \{\alpha_{k}: i \neq k \neq j\}& 7 \\
    & \{\eta\} & \\\hline
   \sigma_{i,j}^{\pm} & \{\alpha_{k}: i \neq k \neq j\} & 8 \\
    & \{\alpha_{k,k+1}: i-1\neq k \neq i, \hspace{3mm} j-1 \neq k \neq j \} & \\
    & \{\nu_{i-1,i+1}^{\mp}, \nu_{j-1,j+1}^{\mp}\} & \\
    & \{\theta_{i,j}\} & \\ \hline
   \eta_{i}^{\pm} & \{\alpha_{i}\} & 9 \\
    & \{\gamma_{k}: i-1 \neq k \neq i\} & \\
    & \sigma_{i-1,i+1}^{\pm} & \\\hline
   \zeta_{i,j}^{\pm} & \{\alpha_{k}: k \notin \{i, \ldots, j\}\} & 9 \\
    & \{\gamma_{i}, \ldots, \gamma_{j-1}\} & \\
    & \{\gamma_{j+1}, \ldots, \gamma_{i-2}\} \text{ (if $i \neq j+2$)} & \\
    & \{\sigma_{i,j}^{\pm}\} \cup \{\veps_{i-1,j}\} & \\ \hline
   \lambda_{i,i\pm2} & \{\alpha_k : i \neq k \neq i \pm 1\} & 8 \\
    & \{\gamma_{i-1}\} \cup \{\gamma_{i+1}, \ldots, \gamma_{i-4}\} \text{ (if $-$)} & \\
    & \{\gamma_{i}\} \cup \{\gamma_{i+3}, \ldots, \gamma_{i-2}\} \text{ (if $+$)} & \\
    & \{\mu_{i, i \pm 2}^{-}\} &
    \\ \hline
\end{array}$$
    \caption{Auxiliary curves used in this section, with all subindices modulo $g$.}
    \label{tab:auxcurvestaubeta}
\end{table}

\begin{figure}[ht]
  \centering
  \includegraphics[height=3.5cm]{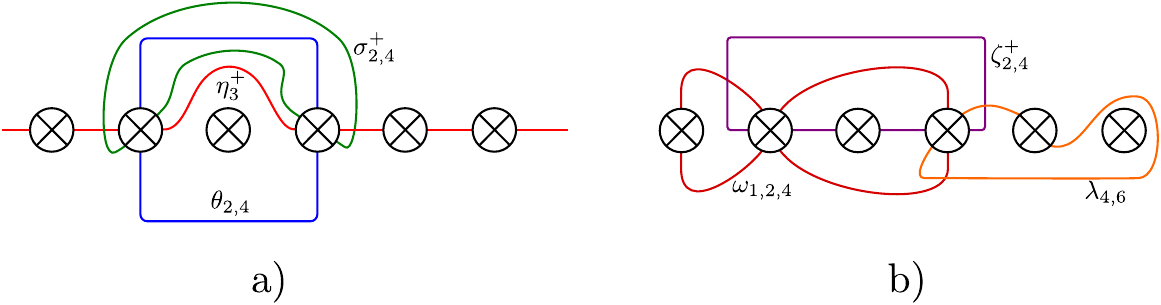}\caption{Examples of curves in Table \ref{tab:auxcurvestaubeta} and the curve $\omega_{1,2,4}$.}\label{Fig-Def-aux-taubeta}
 \end{figure}

On the other hand, let $1 \leq i , j , k \leq g$ be such that $i < j < k$ with respect to the cyclic order in $\Z / g \Z$. We define:

$$\omega_{i,j,k} = \langle \{\alpha_{l}: l \neq i,j,k\} \cup \{\alpha_{j,i-1}, \alpha_{j,k}\} \cup \{\eta\} \rangle \in \Y^{7}.$$

Now, to see that $\tau_{\beta}^{\pm 1} (C_{1}) \subset \Y^{10}$, see Table \ref{tab:taubetaC1}
\begin{table}[]
    $$\begin{array}{|c|c|c|} \hline
      \gamma  &  A \text{ for $\gamma = \langle A \rangle$}  & k \text{ for } \gamma \in \Y^{k}\\ \hline
      \tau_{\beta}(\alpha_{1,2})   & \{\alpha_{5}, \ldots , \alpha_{g}\} \cup \{\alpha_{1,2}\} & 10 \\
       & \{\gamma_{3}\} \cup \{\gamma_{5}, \ldots, \gamma_{g-1}\}\\
       & \{\eta_{3}^{+}, \eta_{4}^{-}\} \cup  \{\veps_{4,g}\} &\\ \hline
      \tau_{\beta}^{-1} (\alpha_{1,2}) & \{\alpha_{5}, \ldots, \alpha_{g}\} & 8 \\
       & \{\gamma_{3}\} \cup \{\eta_{2}^{+}\} \cup \{\zeta_{1,3}^{+}\} &  \\ \hline
      \tau_{\beta}(\alpha_{2,3}) & \{\alpha_{5}, \ldots, \alpha_{g}\} \cup \{\alpha_{2,3}\} & 10 \\
       & \{\theta_{1,4}\} \cup \{\omega_{1,2,3}\} \cup \{\eta_{1}^{-}\} \cup \{\veps_{4,g}\} \\ \hline
      \tau_{\beta}^{-1}(\alpha_{2,3}) & \{\alpha_{5}, \ldots, \alpha_{g}\} \cup \{\alpha_{2,3}\} & 10 \\
       & \{\theta_{1,4}\} \cup \{\omega_{1,2,3}\} \cup \{\eta_{4}^{-}\} \cup \{\veps_{4,g}\} & \\ \hline
      \tau_{\beta}(\alpha_{3,4}) & \{\alpha_{5}, \ldots, \alpha_{g}\} \cup \{\alpha_{3,4}\} & 10 \\
       & \{\gamma_{1}\} \cup \{\gamma_{5}, \ldots, \gamma_{g-1}\} & \\
       & \{\zeta_{2,4}^{+}, \zeta_{g,2}^{-}\} & \\ \hline
      \tau_{\beta}^{-1}(\alpha_{3,4}) & \{\alpha_5 , \ldots, \alpha_g\} & 10 \\
       & \{\gamma_{1}\} \cup \{\omega_{2,3,4}\} \cup \{\theta_{3,5}\} \cup \{\eta_{4}^{+}\} & \\ \hline
      \tau_{\beta}(\alpha_{4,5}) & \{\alpha_{4}, \ldots, \alpha_{g}\} & 10 \\
       & \{\gamma_{1}, \gamma_{2}, \gamma_{4}\} \cup \{\zeta_{2,4}^{+}, \zeta_{g,2}^{-}\} & \\
       & \{\gamma_{6}, \ldots, \gamma_{g-1}\} \text{ (if $g \geq 7$)} & \\ \hline
      \tau_{\beta}^{-1}(\alpha_{4,5}) & \{\alpha_{5}, \ldots, \alpha_{g}\} \cup \{\alpha_{4,5}, \alpha_{4,g}\} & 10 \\
       & \{\gamma_{1}, \gamma_{2}\} \cup \{\delta\} \cup \{\zeta_{g,2}^{-}\} & \\
       & \{\gamma_{6}, \ldots, \gamma_{g-1}\} \text{ (if $g \geq 7$)}
      \\ \hline 
    \end{array}$$
    \caption{$\tau_{\beta}^{\pm 1} (C_{1}) \subset \Y^{10}$.}
    \label{tab:taubetaC1}
\end{table}

Note that in Table \ref{tab:taubetaC1}, there is a curve $\delta$ in $\tau_{\beta}^{-1}(\alpha_{4,5})$; this curve $\delta$ is generated as follows:

$$\delta = \langle \{\alpha_{1}\} \cup \{\alpha_{5}, \ldots, \alpha_{g}\} \cup \{\alpha_{4,5}, \alpha_{4,1}\} \cup \{\gamma_{2}\} \cup \{\gamma_{6}, \ldots, \gamma_{g}\} \cup \{\lambda_{3,1}\} \rangle \in \Y^9$$

\begin{figure}[ht]
  \centering
  \includegraphics[height=3.5cm]{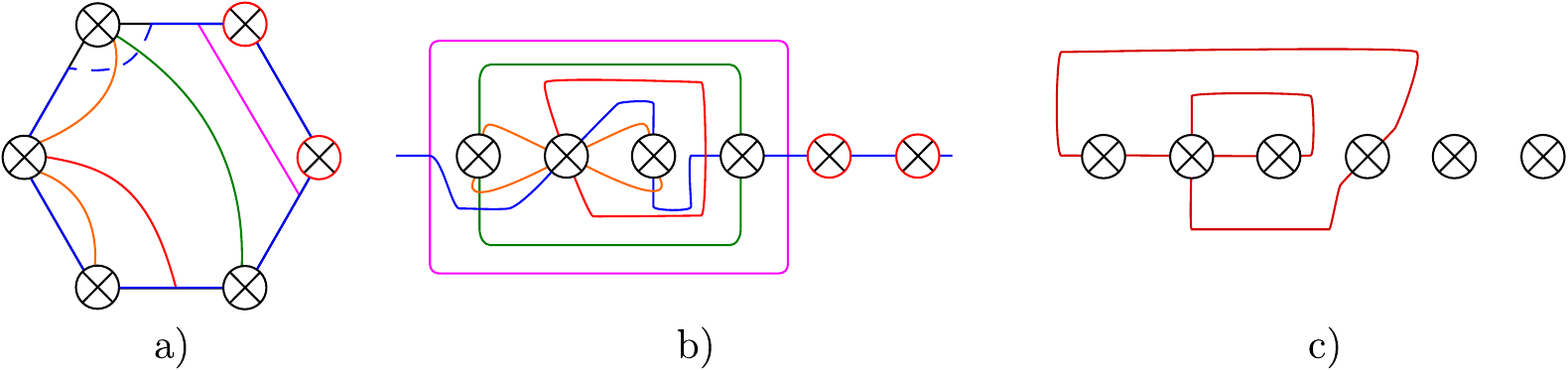}\caption{(a) and (b) depict the curves (from two angles) necessary to uniquely determine $\dt{\beta}(\alpha_{2,3})$), which is depicted in (c).}\label{fig:Def-dtbeta-alpha23}
 \end{figure}

\subsubsection{$\tau_{\beta}^{\pm 1}(C_2) \subset \Y^{12}$}\label{subsubsec:taubetaC2}

Similarly to the previous subsection, we need only to prove that $\tau_{\beta}^{\pm 1}(\alpha_{5,3}) \in \Y^{12}$, due to the same trick with the reflection homeomorphism with axis ``passing through''  the second and third genus.

Note that $\alpha_{5,3} = \langle A_{1} \cup A_{2} \cup A_{3} \rangle$, with $A_{1} = \{\alpha_{2,3}, \alpha_{3,1}, \alpha_{2,g}, \alpha_{1,g-1}\}$, $A_{2} = \{\alpha_{g, g-2}, \ldots, \alpha_{7,5}\} \subset \Y^3$, and $A_3 = \{\omega_{4,5,6}\}$. Then, we have that:
\begin{itemize}
    \item $A_{1} \subset C_1$, and
    \item every curve in $A_{2} \subset \Y^3$ is disjoint from $\beta$.
\end{itemize}

Thus, $\tau_{\beta}^{\pm 1} (\alpha_{5,3}) = \langle \tau_{\beta}^{\pm 1}(A_{1}) \cup \tau_{\beta}^{\pm 1}(A_{2}) \cup \tau_{\beta}^{\pm 1}(A_{3})\rangle$.

Given that $\tau_{\beta}^{\pm 1}(A_1) \subset \Y^{10}$ (due to the previous subsection) and that $\tau_{\beta}^{\pm 1}(A_2) = A_2 \subset \Y^3$, we need only prove that $\tau_{\beta}^{\pm 1}(\omega_{4,5,6}) \in \Y^{11}$.

Again, note that $\omega_{4,5,6} = \langle A_{1} \cup A_{2} \cup A_{3} \rangle$ with $A_{1} = \{\alpha_{2,3}, \alpha_{2,g}, \alpha_{3,1}\}$, $A_{2} = \{\eta\} \cup \{\alpha_{5,6}\} \subset \Y^3$, and either $A_{3}$ empty if $g = 6$, or $A_{3} = \{\alpha_{7}, \ldots, \alpha_{g}\} \subset \Y^6$ if $g \geq 7$.

Thus, $\tau_{\beta}^{\pm 1}(\omega_{4,5,6}) = \langle \tau_{\beta}^{\pm 1}(A_{1}) \cup \tau_{\beta}^{\pm 1}(A_{2}) \cup \tau_{\beta}^{\pm 1} (A_{3})\rangle$. Since $\tau_{\beta}^{\pm 1}(A_{1}) \subset \Y^{10}$ (due to the previous subsection) and $\tau_{\beta}^{\pm 1}(A_{2} \cup A_{3}) = A_2 \cup A_3 \subset \Y^6$, we have that $\tau_{\beta}^{\pm 1}(\omega_{4,5,6}) \in \Y^{11}$.

\subsubsection{$\tau_{\beta}^{\pm 1}(C_3) \subset \Y^{13}$}\label{subsubsec:taubetaC3}

If $g \geq 7$, $C_{3}$ is empty. Hence in this subsection we assume that $g = 6$, and thus $C_{3} = \{\nu_{5,7}^{-}\}$.

Since $\tau_{\beta}^{\pm 1}(\nu_{5,7}^{-})$ are as in Figure \ref{Fig-dtbeta-nu57} and the first author is deadly afraid of said curves, we employ the same strategy as in the previous subsection. For this, we first prove that $\tau_{\beta}^{\pm 1}(\{\alpha_{1}, \ldots, \alpha_{4}\}) \subset \Y^{11}$ and $\tau_{\beta}^{\pm 1}(\nu_{3,6}^{+}) \in \Y^{12}$.

\begin{figure}[ht]
  \centering
  \includegraphics[height=3.5cm]{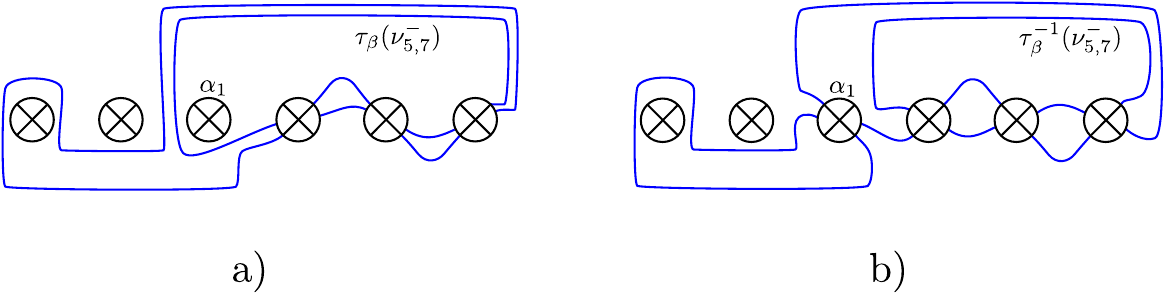}\caption{The curves $\tau_{\beta}(\nu_{5,7}^{-})$ and $\tau_{\beta}^{-1}(\nu_{5,7}^{-})$.}\label{Fig-dtbeta-nu57}
 \end{figure}

\begin{itemize}
    \item $\tau_{\beta}^{\pm 1}(\{\alpha_{1}, \ldots, \alpha_{4}\}) \subset \Y^{11}$: For each $1 \leq i \leq 4$, there exist sets $A_{1} \subset C_{1}$ and $A_{2}$ such that $\alpha_{i} = \langle A_{1} \cup A_{2} \rangle$ and every curve in $A_{2}$ is disjoint from $\beta$. See Table \ref{tab:alphaiA1A2}.
    \begin{table}[h]
        $$\begin{array}{|c|c|c|}\hline
            i & A_{1} & A_{2} \\ \hline
            1 & \alpha_{3,1}, \alpha_{2,3}, \alpha_{3,4}, \alpha_{4,5} & \alpha_{5,6}, \alpha_{6,4} \\ \hline
            2 & \alpha_{2,g-1}, \alpha_{4,2}, \alpha_{3,4} & \alpha_{5,6}, \alpha_{6,4}, \theta_{1,3} \\ \hline
            3 & \alpha_{1,2}, \alpha_{1,g-1}, \alpha_{2,g} & \alpha_{5,6}, \alpha_{6,4}, \theta_{2,3} \\ \hline
            4 & \alpha_{2,3}, \alpha_{1,g-1}, \alpha_{2,g}, \alpha_{3,1} & \alpha_{5,6}, \alpha_{6,4} 
            \\ \hline
        \end{array}$$
        \caption{Sets $A_{1}$ and $A_{2}$ needed for $\alpha_{i} = \langle A_{1} \cup A_{2} \rangle$, for Subsection \ref{subsubsec:taubetaC3}.}
        \label{tab:alphaiA1A2}
    \end{table}
    \item $\tau_{\beta}^{\pm 1}(\nu_{3,6}^{+}) \in \Y^{12}$: Let $\delta = \langle \{\alpha_{3}, \ldots, \alpha_{6}\} \cup \{\gamma_{1}\} \cup \{\eta_{5}^{+}\} \rangle \in \Y^{10}$; see Figure \ref{Fig-Def-dtbeta-nu36}. Then, $\nu_{3,6}^{+} = \langle A_{1} \cup A_{2} \rangle$ with $A_{1} = \{\alpha_{1}, \ldots, \alpha_{4}\} \cup \{\alpha_{4,5}\}$ and $A_{2} = \{\alpha_{5}, \alpha_{6}, \delta\}$.\\ This implies that $\tau_{\beta}^{\pm 1}(\nu_{3,6}^{+}) = \langle \tau_{\beta}^{\pm 1}(A_{1}) \cup \tau_{\beta}^{\pm 1}(A_{2})\rangle = \langle \tau_{\beta}^{\pm 1}(A_{1}) \cup A_{2} \rangle \in \Y^{12}$.
\end{itemize}

Now, we have that $\nu_{5,7}^{-} = \langle A_{1} \cup A_{2} \cup A_{3} \rangle$ with $A_{1} = \{\alpha_{3,1}, \alpha_{4,2}, \alpha_{2,3}, \alpha_{3,4}\} \subset C_1$, $A_{2} = \{\alpha_{1}, \nu_{3,6}^{+}\}$ and $A_{3} = \{\alpha_{5}, \alpha_{6}\}$. Thus, $\tau_{\beta}^{\pm 1}(\nu_{5,7}^{-}) = \langle \tau_{\beta}^{\pm 1}(A_1) \cup \tau_{\beta}^{\pm 1}(A_2) \cup \tau_{\beta}^{\pm 1}(A_3) \rangle = \langle \tau_{\beta}^{\pm 1}(A_{1}) \cup \tau_{\beta}^{\pm 1}(A_2) \cup A_3 \rangle$. By the previous subsection and the argument above, we have that $\tau_{\beta}^{\pm 1}(A_{1}) \subset \Y^{10}$ and $\tau_{\beta}^{\pm 1}(A_{2}) \subset \Y^{12}$. Therefore $\tau_{\beta}^{\pm 1}(\nu_{5,7}^{-}) \in \Y^{13}$.

\begin{figure}[ht]
  \centering
  \includegraphics[height=3.5cm]{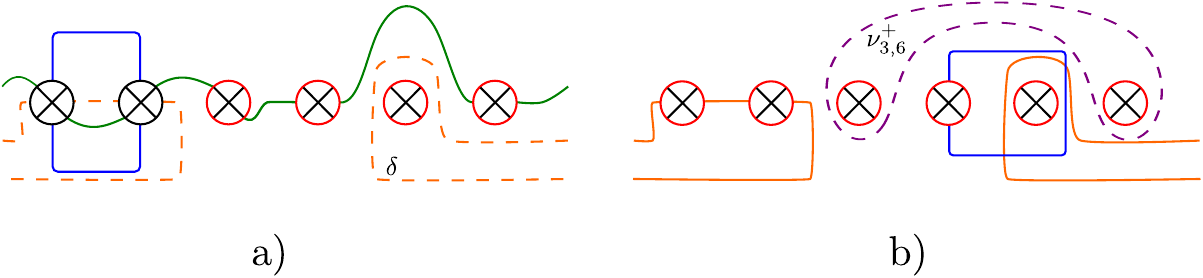}\caption{(a) depicts the curve $\delta$ (orange and dotted) and the curves needed to uniquely determine it; (b) depicts the curve $\nu_{3,6}^{+}$ (purple and dotted) and the curves needed to uniquely determine it.}\label{Fig-Def-dtbeta-nu36}
 \end{figure}


\section{Proof of \Cref{thm:exhaustionY}}\label{sec:thmexhaustion}

Given an essential curve $\gamma$ in a non-orientable surface $N$, let $N_{\gamma}$ be the surface obtained by cutting $N$ along $\gamma$.

\begin{proof}[Proof of \Cref{thm:exhaustionY}]
Let $N$ be a connected non-orientable surface with genus $g\geq 6$ and let $\gamma$ be an essential curve. 

Suppose $g$ is odd. If $\gamma$  is a one-sided curve, using the Euler characteristic, we have that $N_{\gamma}$ is an orientable surface of genus $\frac{g-1}{2}$ with one boundary component or a non-orientable surface of genus $g-1$ with one boundary component. If $\gamma$ is a two-sided non-separating curve, we have that $N_{\gamma}$ is non-orientable of genus $g-2$. This implies that under the action of $\mcg{N}$ in $\cg{N}$ we have two classes of one-sided curves with representatives $\alpha_{1,2} \in V(\Y)$ and $\eta \in V(\Y^{6})$ (see paragraph before Lemma \ref{lemma:dtgammaiAp2}) and we have one class of a two-sided non-separating curve with a representative $\gamma_{1}\in V(\Y^{5})$ (see Lemma \ref{lemma:gammai}).

Analogously, if $g$ is even, under the action of $\mcg{N}$ in $\cg{N}$ we have one class of one-sided curves with a representative  $\alpha_{1,2} \in V(\Y)$ and two classes of two-sided non-separating curves with representatives $\gamma_1 \in V(\Y^{5})$ and $\eta\in V(\Y^{6})$.
Then by Lemma \ref{lemma:groupYinY} and \Cref{sec:GenSetY} we have that for any non-separating curve $\gamma \in V(\cg{N})$ there exist $k\in\N$ such that $\gamma \in V(\Y^{k})$, in particular every non-separating curve is an element of $\bigcup_{i\in \N} \Y^{i}$.

Now let $\gamma$ be a separating curve, we have that every such a curve can be uniquely determined by a finite set $C$ of nonseparating curves (we can fill each component of $N_{\gamma}$ with nonseparating curves). We have that $C\subseteq V(\Y^{k})$ for some $k\in \N$, thus $\gamma\in V(\Y^{k+1})$. Therefore, $\cg{N} = \bigcup_{i\in \N}\Y^{i}$.
\end{proof}

Let $B_{1}$, $B_{2}$  and $\X$ the following sets:
\begin{align*}
B_1 \, = & \, \, \{ \alpha_{i} \mid 1\leq i \leq g \}\, \cup \, \{\alpha_{i,j} \mid 1\leq i,j\leq g, \,\, j\neq i, \,\, j \neq i-1   \}\\ 
& \,\, \cup \,  \{\varepsilon_{i,j} \mid  1\leq i,j\leq g, \,\, 2 \leq |i-j| \leq g-2  \},
\end{align*}
$$
B_{2}\, = \,\{\nu^{+}_{i,i+2},\, \nu^{-}_{i,i+2} \mid 1\leq i \leq g \,\, \text{ and } \, i+2 \text{ modulo } g  \},
$$
$$
\X = B_1\cup B_2.
$$
Irmak shows that $\X$ is a finite  rigid set with trivial pointwise stabilizer (see Lemma 3.6 in \cite{Irmak2019}).

\begin{proof}[Proof of Corollary \ref{cor:exhaustionX}]
 By construction $\Y\subseteq \X$, then  $\Y^{k} \subseteq \X^{k}$ for any  $k\in \N$. This implies that $\bigcup_{i\in \N}\Y^{i} = \bigcup_{i\in \N}\X^{i}$, then $\X$ is a finite rigid seed subgraph of $\cg{N}$.    
\end{proof}

\bibliographystyle{plain}
\bibliography{bibliography}

\begin{thebibliography}{10}

\bibitem{AL13}
Javier Aramayona and Christopher~J. Leininger.
\newblock Finite rigid sets in curve complexes.
\newblock {\em J. Topol. Anal.}, 5(2):183--203, 2013.

\bibitem{AL16}
Javier Aramayona and Christopher~J. Leininger.
\newblock Exhausting curve complexes by finite rigid sets.
\newblock {\em Pacific J. Math.}, 282(2):257--283, 2016.

\bibitem{Atalan-Korkmaz}
Ferihe Atalan and Mustafa Korkmaz.
\newblock Automorphisms of curve complexes on nonorientable surfaces.
\newblock {\em Groups Geom. Dyn.}, 8(1):39--68, 2014.

\bibitem{BM06}
Jason Behrstock and Dan Margalit.
\newblock Curve complexes and finite index subgroups of mapping class groups.
\newblock {\em Geom. Dedicata}, 118:71--85, 2006.

\bibitem{FarbMargalit}
Benson Farb and Dan Margalit.
\newblock {\em A primer on mapping class groups}, volume~49 of {\em Princeton
  Mathematical Series}.
\newblock Princeton University Press, Princeton, NJ, 2012.

\bibitem{Harvey81}
W.~J. Harvey.
\newblock Boundary structure of the modular group.
\newblock In {\em Riemann surfaces and related topics: {P}roceedings of the
  1978 {S}tony {B}rook {C}onference ({S}tate {U}niv. {N}ew {Y}ork, {S}tony
  {B}rook, {N}.{Y}., 1978)}, volume No. 97 of {\em Ann. of Math. Stud.}, pages
  245--251. Princeton Univ. Press, Princeton, NJ, 1981.

\bibitem{JHH2}
Jes\'us Hern\'andez~Hern\'andez.
\newblock Edge-preserving maps of curve graphs.
\newblock {\em Topology Appl.}, 246:83--105, 2018.

\bibitem{JHH1}
Jes\'{u}s Hern\'{a}ndez~Hern\'{a}ndez.
\newblock Exhaustion of the curve graph via rigid expansions.
\newblock {\em Glasg. Math. J.}, 61(1):195--230, 2019.

\bibitem{JHH3}
Jes\'us Hern\'andez~Hern\'andez.
\newblock Graph morphisms and exhaustion of curve graphs of low-genus surfaces.
\newblock {\em Proc. Edinb. Math. Soc. (2)}, 68(4):1018--1068, 2025.

\bibitem{Irmak2004}
Elmas Irmak.
\newblock Superinjective simplicial maps of complexes of curves and injective
  homomorphisms of subgroups of mapping class groups.
\newblock {\em Topology}, 43(3):513--541, 2004.

\bibitem{Irmak2012}
Elmas Irmak.
\newblock Superinjective simplicial maps of the complexes of curves on
  nonorientable surfaces.
\newblock {\em Turkish J. Math.}, 36(3):407--421, 2012.

\bibitem{Irmak2020}
Elmas Irmak.
\newblock Edge-preserving maps of the nonseparating curve graphs, curve graphs
  and rectangle preserving maps of the {H}atcher-{T}hurston graphs.
\newblock {\em J. Knot Theory Ramifications}, 29(11):2050078, 41, 2020.

\bibitem{Irmak2019}
Elmas Irmak.
\newblock Exhausting curve complexes by finite rigid sets on nonorientable
  surfaces.
\newblock {\em J. Topol. Anal.}, 16(2):261--289, 2024.

\bibitem{Ivanov}
Nikolai~V. Ivanov.
\newblock Automorphism of complexes of curves and of {T}eichm\"{u}ller spaces.
\newblock {\em Internat. Math. Res. Notices}, (14):651--666, 1997.

\bibitem{Korkmaz}
Mustafa Korkmaz.
\newblock Automorphisms of complexes of curves on punctured spheres and on
  punctured tori.
\newblock {\em Topology Appl.}, 95(2):85--111, 1999.

\bibitem{Lickorish}
W.~B.~R. Lickorish.
\newblock Homeomorphisms of non-orientable two-manifolds.
\newblock {\em Proc. Cambridge Philos. Soc.}, 59:307--317, 1963.

\bibitem{Luo}
Feng Luo.
\newblock Automorphisms of the complex of curves.
\newblock {\em Topology}, 39(2):283--298, 2000.

\bibitem{Shackleton2007}
Kenneth~J. Shackleton.
\newblock Combinatorial rigidity in curve complexes and mapping class groups.
\newblock {\em Pacific J. Math.}, 230(1):217--232, 2007.

\bibitem{SS25}
{Micha\l} Stukow and {B\l}a\.{z}ej Szepietowski.
\newblock Automorphisms of the pants graph of a nonorientable surface.
\newblock Preprint at
  \href{https://doi.org/10.48550/arXiv.2507.12613}{arXiv:2507.12613 [math.GT]}.

\bibitem{Szepietowski20}
B\l~a\.zej Szepietowski.
\newblock A note on the curve complex of the 3-holed projective plane.
\newblock {\em Math. Commun.}, 25(2):289--296, 2020.

\bibitem{Szepietowski}
{B\l}a\.{z}ej Szepietowski.
\newblock The mapping class group of a nonorientable surface is generated by
  three elements and by four involutions.
\newblock {\em Geom. Dedicata}, 117:1--9, 2006.

\end{thebibliography}
\end{document}